\documentclass[microtype,reqno]{gtpart}     
\usepackage{graphicx}
\usepackage{grffile}
\usepackage[pointlessenum]{paralist}%
\usepackage{mdwlist}
\usepackage[mathscr]{euscript}%

\newcommand\figwidth{0.95\textwidth}

\title[Some $3$--dimensional transverse $\C$--links]
{Some $\mathbf{3}$--dimensional transverse $\C$--links\\
(Constructions of higher-dimensional $\C$--links, I)}%

\author[Lee Rudolph]{Lee Rudolph}
\givenname{Lee}
\surname{Rudolph}
\address{44 Allandale Street, Apt.\ 302
\\Jamaica Plain MA 02130\\USA\\}
\email{lrudolph@meganet.net}





\arxivreference{}
\arxivpassword{}

\volumenumber{}
\issuenumber{}
\publicationyear{}
\papernumber{}
\startpage{}
\endpage{}
\doi{}
\MR{}
\Zbl{}
\received{}
\revised{}
\accepted{}
\published{}
\publishedonline{}
\proposed{}
\seconded{}
\corresponding{}
\editor{}
\version{}

\swapnumbers
\newtheorem{theorem}{Theorem}[section]
\newtheorem{proposition}[theorem]{Proposition}
\newtheorem{corollary}[theorem]{Corollary}
\newtheorem{lemma}[theorem]{Lemma}

\newtheorem*{theorem*}{Theorem}
\newtheorem*{lemma*}{Lemma}
\newtheorem*{corollary*}{Corollary}
\newtheorem*{proposition*}{Proposition}
\theoremstyle{definition}

\newtheorem{definition}[theorem]{Definition}
\newtheorem{definitions}[theorem]{Definitions}
\newtheorem*{definition*}{Definition}
\newtheorem*{definitions*}{Definitions}
\newtheorem{example}[theorem]{Example}

\newtheorem{question}[theorem]{Question}
\newtheorem*{question*}{Question}
\newtheorem{questions}[theorem]{Questions}
\newtheorem*{questions*}{Questions}
\newtheorem{remark}[theorem]{Remark}
\newtheorem{remarks}[theorem]{Remarks}
\newtheorem*{remark*}{Remark}
\newtheorem*{remarks*}{Remarks}
\newtheorem{notation}[theorem]{Notation}
\newtheorem*{notation*}{Notation}

\input{mdefs.mac}

\begin{document}

\begin{abstract}   
By use of a variety of techniques (most based on 
constructions of quasipositive knots and links, some old 
and others new), many smooth $3$--manifolds are 
realized as transverse intersections of complex 
surfaces in $\C^3$ with strictly pseudoconvex $5$--spheres.  
These manifolds not only inherit interesting 
intrinsic structures (eg, they have canonical Stein-fillable 
contact structures), they also have extrinsic structures 
of a knot-theoretical nature (eq, $S^3$ arises in infinitely 
many distinct ways).  This survey is not 
comprehensive; a number of questions are left open for future work.
\end{abstract}
\begin{asciiabstract}
By use of a variety of techniques (most based on 
constructions of quasipositive knots and links, some old 
and others new), many smooth 3-manifolds are 
realized as transverse intersections of complex 
surfaces in complex 3-space with strictly pseudoconvex 
5-spheres.  These manifolds not only inherit interesting 
intrinsic structures (eg, they have canonical Stein-fillable 
contact structures), they also have extrinsic structures 
of a knot-theoretical nature (eq, the 3-sphere arises in 
infinitely many distinct ways).  This survey is not 
comprehensive; a number of questions are left open for 
future work.
\end{asciiabstract}
\maketitle


\section{Introduction\label{part:introduction}}

A \bydef{$k$--dimensional link} in a smooth, oriented 
$m$--manifold $M$ is a pair $\Lscr=(L,M)$ where 
$L\sub M$ is a compact, non-empty, purely $k$--dimensional
manifold (without boundary) called the \bydef{link-manifold} 
of $\Lscr$; $\Lscr$ is \bydef{classical} when $k=1$, $m=3$,
and $M$ is diffeomorphic to $S^3$.  
In case $L$ is endowed with an extra structure (such as being 
smooth), $\Lscr$ is also said to have that structure.  
A \bydef{knot} is a link with connected link-manifold.

For $n\ge 1$, let $\Sigma\sub\C^{n+1}$ 
be a strictly pseudoconvex $(2n+1)$--sphere, 
$\Delta\sub\C^{n+1}$ the closed Stein $(2n+2)$--disk 
it bounds, and $U$ an open Stein neighborhood of $\Delta$ 
in $\C^{n+1}$. If $f\in\holo{U}$ is a non-constant 
holomorphic function without repeated factors, 
then $V(f)\isdefinedas f^{-1}(0)$ is a complex-analytic 
hypersurface in $U$; up to multiplicities, every 
complex-analytic hypersurface in $U$ has the form $V(f)$.  
Let $\lman{f}{\Sigma}\isdefinedas V(f)\cap\Sigma$,
$\Cspan{f}{\Delta}\isdefinedas V(f)\cap\Delta$.

\begin{definitions}
\begin{inparaenum}[(1)]
\item
Suppose that the singular set $\sing{V(f)}$ of $V(f)$
has empty intersection with $\Sigma$, so that
$\lman{f}{\Sigma}$ is the intersection of $\Sigma$ with the complex $n$--manifold $\reg{V(f)}$ of regular points of $V(f)$.  
If this intersection is transverse, then $\lman{f}{\Sigma}$ 
is a smooth compact $(2n-1)$--manifold.  In case
either 
	\begin{inparaenum}[(a)]\
	\item
	$n>0$ and $\lman{f}{\Sigma}\ne\emptyset$, or
	\item
	$n=0$ (so that necessarily $\lman{f}{\Sigma}=\emptyset$)
	and $\Cspan{f}{\Delta}\ne\emptyset$,
	\end{inparaenum}
call the smooth link $\Clink{f}{\Sigma}%
\isdefinedas(\lman{f}{\Sigma},\Sigma)$ a 
$(2n-1)$--dimensional \bydef{transverse $\C$--link}.
\item
In case $\lman{f}{\Sigma}$ is a compact $(2n-1)$--manifold
that is \textit{not} smoothly embedded in $\Sigma$, call
$\Clink{f}{\Sigma}$ a $(2n-1)$--dimensional 
\bydef{wild $\C$--link}.
\end{inparaenum} \hfill\done
\end{definitions}

\begin{remarks}\label{rmks:C-link terminology}
\begin{inparaenum}[(1)]
\item\label{rmk:C-link history}
The term ``$\C$--link'' was introduced (Rudolph 
\cite{Rudolph2005}) as a way to include under one name
two types of classical links which share the defining
feature that, up to ambient isotopy, they arise as intersections
of a complex plane curve $V\sub\C^2$ with a $3$--sphere $\Sigma\sub\C^2$.  One of these types is the special case 
in that dimension of transverse $\C$--links.  The other type,
``totally tangential $\C$--links'', also can be generalized
to higher dimensions but will be left undefined here 
(and will be ignored except in a small neighborhood of
Question \ref{qn:generalized qp annuli}).  
There are no $1$--dimensional wild $\C$--links.
\item\label{rmk:empty C-links}
For $q\ge 1$, let $f_q\from\C\to\C\mapsuchthat z\mapsto z^q$.  
The empty $(-1)$--dimensional $\C$--link
$\Clink{f_q}{S^1}=(\emptyset,S^1)\defines [q]$ is endowed 
with an extra structure---namely, the degree--$q$ fibration 
$f_q\restr S^1\from S^1\to S^1$---that makes $[q]$ a
degenerate but very useful \bydef{fibered link} 
as defined and discussed in \ref{subsubsect:fibered links, 
open books, and contact structures}.  The links and
notation $[q]$ were introduced by Kauffman and 
Neumann \cite{KauffmanNeumann1977} for an application 
to be used in \ref{subsect:cyclic branched covers (general)}.
\end{inparaenum}
\done
\end{remarks}

In this paper I launch investigations into 
$3$--dimensional transverse $\C$--links 
(for some remarks on wild $\C$-links in
odd dimensions greater than or equal to $3$,
see Rudolph \cite{Rudolph2015}). 
I describe in more or less detail several  
constructions of such links and a few of their 
interesting properties; deeper investigations 
are deferred to a later date.  Most of the 
new $3$--dimensional constructions rely, in turn, on 
various constructions of $1$--dimensional transverse 
$\C$--links---some of them new, and presented
with proofs or proof sketches, and others simply 
restated (with references to published proofs) as 
needed.

One important special case of $3$--dimensional transverse 
$\C$--links is well known and well understood.  Let 
$U$ be a neighborhood of $\zz\in\C^3$, $f\in\holo{U}$.
If $\zz$ is an isolated singular point (or a regular point)
of $V(f)$, then for all sufficiently 
small $\epsilon>0$ the (round) $5$--sphere 
$\Sigma=\SPH{5}{\zz}{\epsilon}$ intersects
$\reg{V(f)}$ transversely.  The ambient isotopy type of 
the transverse $3$--dimensional $\C$--link 
$\Clink{f}{\Sigma}$ is independent of $\epsilon$; any 
representative of this ambient isotopy type is called 
the \bydef{link of the isolated singular point} of $f$ 
at $\zz$, and may be denoted 
$\Clinksing{\zz}{f}=(\lman{f}{\zz},S^5)$. 
Milnor \cite{Milnor1968} began the systematic 
knot-theoretical study of these links (and their analogues 
in higher dimensions; classical knot theory had been 
applied to singular points of complex plane curves since
Brauner \cite{Brauner1928}), and there is now a huge
body of research on the topology (and geometry) both of 
their link-manifolds and of the embeddings of those 
link-manifolds in their ambient $5$--spheres.

A second special case of $3$--dimensional transverse 
$\C$--links (and their analogues in other dimensions, 
including the classical) has also been studied, though 
less thoroughly. If $f\in\holo{\C^3}$ is a complex polynomial 
function and $V(f)$ is a finite set, then 
for all sufficiently small $\epsilon>0$ the (round) $5$--sphere $\Sigma=\SPH{5}{\zero}{1/\epsilon}$ intersects $\reg{V(f)}$ transversely. The ambient isotopy type of the transverse 
$3$--dimensional $\C$--link $\Clink{f}{\Sigma}$ is 
independent of $\epsilon$; any representative of this ambient 
isotopy type is called the \bydef{link at infinity} of $f$, 
and may be denoted $\Clinkinfty{f}=(\lman{f}{\infty},S^5)$.  
The link at infinity of a complex algebraic plane curve 
was introduced under that name by Rudolph \cite{Rudolph1982}, 
though implicit earlier in Chisini 
\cite{Chisini1933,Chisini1937}; links at infinity of complex 
algebraic hypersurfaces in all dimensions were introduced 
by Neumann and Rudolph 
\cite{NeumannRudolph1987,NeumannRudolph1988}. See Rudolph \cite{Rudolph2005} for further references.

Aside from those two special cases, very little is known
(or has been published) about $3$--dimensional transverse
$\C$--links.  This contrasts considerably with the situation
for $1$--dimensional transverse $\C$--links, where---by
taking Boileau and Orevkov \cite{BoileauOrevkov2001}
(applying Eliashberg \cite{Eliashberg1990})
and Rudolph \cite{Rudolph1983} together---such links 
are known to be (up to ambient isotopy) precisely the \bydef{quasipositive links}.  This characterization is 
not effective, in the sense that no algorithm is presently 
known to determine whether or not a given smooth, 
oriented classical link is quasipositive (precisely: the class 
of quasipositive links is recursive, but is not known 
to be recursively enumerable).  However, there is an 
abundance of ways to construct quasipositive links with
various prescribed properties, as can be seen in the next
part of this paper.

It is not clear whether or to what extent the notion 
of ``quasipositive link'' has useful generalizations 
in higher dimensions, much less whether for some such 
generalization(s) there exist analogues of 
\cite{Rudolph1983} and \cite{BoileauOrevkov2001} that
could lead to a topological characterization of higher-dimensional
transverse $\C$--links.  Even for specifically $3$--dimensional
transverse $\C$--links, this may be a daunting task. 
For the purposes of the following brief and speculative discussion,  
definitions of terminology not already introduced can be 
found in the preliminaries, section \ref{sect:preliminaries}.

First note that in the case of a $1$--dimensional transverse
$\C$--link $\Clink{f}{\Sigma}$ (assuming it to be generic, 
ie, such that $V(f)\cap\Delta$ is non-singular) there 
is nothing special about the 
intrinsic topology of $\lman{f}{\Sigma}$ or 
$\Cspan{f}{\Delta}$:
any non-empty compact oriented $1$--manifold without boundary 
occurs as $\lman{f}{S^3}$, and any compact oriented 
$2$--manifold without closed components appears as 
$\Cspan{f}{D^4}$, for some $\Clink{f}{D^4}$.  This contrasts
with the situation for $3$--dimensional transverse $\C$--links 
in a strictly pseudoconvex $6$--disk $\Delta$: if 
$\Clink{f}{\Delta}$ is generic, then $\Cspan{f}{\Delta}$ is 
a compact Stein manifold-with-boundary bounded by 
$\lman{f}{\Sigma}$, and both kinds of spaces are subject
to non-trivial topological restrictions.  By \cite{BoileauOrevkov2001}, an intrinsic 
topological restriction on $\Cspan{f}{\Delta}$ (due to 
Loi and Piergallini \cite{LoiPiergallini2001}, with 
later proofs by Akbulut and Ozbagci \cite{AkbulutOzbagci2001}
and Giroux \cite{Giroux2002}) can be restated thus. 

\begin{theorem}\label{thm:Loi-Piergallini theorem}
A compact, oriented, smooth $4$--manifold-with-boundary
$W$ is diffeomorphic to a compact Stein surface iff 
$W$ is a branched covering of $D^4$ over the $\C$--span
$\Cspan{f}{D^4}$ of a $1$--dimensional transverse $\C$--link
$\Clink{f}{S^3}$. \qed
\end{theorem}

As noted by Etnyre \cite{Etnyre2002}, it is a corollary
to Loi and Piergallini's \fullref{thm:Loi-Piergallini theorem}
that the intrinsic topology of $\lman{f}{\Sigma}$ is
restricted as follows.
\begin{corollary}\label{cor:Etnyre's observation}
A compact, oriented $3$--manifold $M$ is diffeomorphic
to the (strictly pseudoconvex) boundary of a compact
Stein surface iff there exists an 
open book $\book{b}\from M\to \C$ which is \emph{positive} 
in the sense that its geometric monodromy 
$F_0(\book{b})\to F_0(\book{b})$ 
(where $F_0(\book{b})\isdefinedas \invof{\book{b}}([0,\infty{[})$,
the ``first page'' of $\book b$, is a compact oriented 
$2$--manifold-with-boundary) can be written as a product
of positive Dehn twists on $F_0(\book{b})$.\qed
\end{corollary}

Moreover, $W$ can be reconstructed from its boundary $M$
together with such a positive factorization of the monodromy
of an open book on $M$ (different open books, or even different
factorizations, may give different $4$--manifolds $W$).

However, neither of these necessary conditions on $W$ 
and $M=\Bd W$ is sufficient to ensure that they actually
occur as $\Cspan{f}{\Delta}$ or $\lman{f}{\Sigma}$.  In fact,
although every Stein surface embeds properly and holomorphically 
in $\C^4$ (Eliashberg and Gromov \cite{EliashbergGromov1992},
Sch\"urmann \cite{Schuermann1992,Schuermann1997}), there are 
compact Stein surfaces that do not embed holomorphically in 
$\C^3$ (indeed, whose underlying differentiable manifolds 
do not embed smoothly in $\R^6$; Forster \cite{Forster1970}).
Suppose, however, that $W$ is in fact a compact Stein surface
embedded holomorphically (with strictly pseudoconvex boundary)
in $\C^3$.  In this case, it is easy (possibly after slightly
perturbing the complex structure) actually to embed $W$ as a 
Stein domain on a (non-singular) complex algebraic surface 
in $V(f)\sub\C^3$; but it is not immediately obvious that this
can be done in such a way that $W=\Cspan{f}{\Delta}$ for
some strictly pseudoconvex $6$--disk $\Delta$ in $\C^3$. 

\begin{questions}\label{qn:which Stein surfaces are Cspans?}
\begin{inparaenum}[(1)]
\item\label{subqn:which Stein surfaces embed in C3?}
What are necessary and/or sufficient conditions on a
compact, oriented, smooth $4$--manifold with boundary 
that it be diffeomorphic to a compact Stein surface 
embedded holomorphically (with strictly pseudoconvex
boundary) in $\C^3$?
\item\label{subqn:which Stein surfaces in C3 are Cspans?} 
What are necessary and/or sufficient conditions on
a compact Stein surface embedded holomorphically 
(with strictly pseudoconvex boundary) in $\C^3$ that 
it be $\Cspan{f}{\Delta}$ for some strictly pseudoconvex 
$6$--disk $\Delta\sub\C^3$? 
\end{inparaenum}
\done
\end{questions}

The extreme (not to say pathological) behavior exhibited
by Stein domains in $\C^2$, as demonstrated by Gompf
\cite{Gompf2013}, suggests that any full answer to these
questions may be quite alarming.

Conceivably the following question can be more easily 
answered.
\begin{question}\label{qn:which 3mflds are Clink-manifolds?}
What are necessary and/or sufficient conditions on a
compact, oriented $3$--manifold $M$ that it support some 
positive open book $\book{b}\from M\to\C$ that is associated
to a realization of $M$ as a link-manifold $\lman{f}{\Sigma}$?
\done
\end{question}

Another question, presumably easier than characterizing 
all $3$--dimensional transverse $\C$--links (whether or 
not by generalizing quasipositive links), is the following.

\begin{question}\label{qn:generalize strong quasipositivity}
Can some non-trivial family of $3$--dimensional $\C$--links
be characterized by a reasonable generalization of the 
notion of a strongly quasipositive link (see 2.4)?\done
\end{question}

At a minimum, such a generalization would presumably involve
finding properties---including, but going further than, the
topological condition in Theorem \ref{thm:Loi-Piergallini theorem}%
---that a $4$--dimensional submanifold-with-boundary $W$ 
of a strictly pseudoconvex $5$--sphere $\Sigma=\Bd\Delta\sub\C^3$
must possess for there to be an ambient isotopy carrying 
some $3$--dimensional transverse $\C$--link $\Clink{f}{\Sigma}$ 
onto $(\Bd W,\Sigma)$ and $(\Cspan{f}{\Delta},\Delta)$ 
onto $(W',\Delta)$, where $W'$ is obtained from $W$ by 
leaving $\Bd W$ fixed and pushing $\Int W$ into $\Int\Delta$.  

Question \ref{qn:generalize strong quasipositivity}
can be made more particular yet.  The family of
strongly quasipositive $2$--component links 
$\Clink{f}{S^3}$ such that $\Cspan{f}{D^4}$ is an annulus
can be characterized as those that can be obtained 
(via a digression into $1$--dimensional totally tangential
$\C$--links, Rudolph \cite{Rudolph1992a,Rudolph1992,Rudolph1998}) using a real-analytic Legendrian simple closed curve in $S^3$ 
(with its standard contact structure) and its canonical framing; 
see Theorem \ref{thm:qp annuli}%
\eqref{subthm:qp knotted annuli and TB}.

\begin{question}\label{qn:generalized qp annuli}
Is there a reasonable generalization of strongly 
quasipositive annuli?\done
\end{question}
It is easy to construct totally tangential $2$--dimensional
$\C$--links in $S^5$---which are, in particular, Legendrian
manifolds---diffeomorphic to $S^2$ and $S^1\times S^1$;
and these do give $3$--dimensional transverse $\C$--links
(with link-manifolds $S^2\times S^1$ and $(S^1)^3$, respectively,
for the examples I have in mind).  It may well be possible,
if not easy, to do the same for $2$--manifolds $F_g$ of 
genus $g>1$; a careful reading of Haskins and Kapouleas
\cite{HaskinsKapouleas2007} might even provide an appropriate 
reader (which I am not) with the affirmative answer.

\begin{remark} A recent theorem of Kasuya \cite{Kasuya2014}
implies that if $\Clink{f}{\Sigma}$ is a $3$--dimensional 
$\C$--link then the first Chern class of its link-manifold $\lman{f}{\Sigma}$ vanishes.\done
\end{remark}

\section{Old and new constructions of quasipositive links%
\label{part:qp stuff old and new}}
This part of the paper assembles constructions of 
quasipositive links used in the next part to construct 
$3$--dimensional $\C$--links.  For further information, 
particularly about constructions not flagged as either new or
incorporating new details, see \cite{Rudolph2005} and sources 
cited there.

\subsection{Preliminaries on braids, 
plumbing,
trees, 
fibered links, 
etc\label{sect:preliminaries}}
For general material on braids and closed braids (as well as
plats, used in passing in \ref{subsubsect:qp rational links}),
see Birman \cite{Birman1974} or Birman and Brendle
\cite{BirmanBrendle2005}.  For details and further 
references on braided surfaces, quasipositive braids, 
etc, see \cite{Rudolph2005}.  

For an historical survey of open books, see 
Winkelnkemper \cite{Winkelnkemper1998}.  For
details and further references on contact structures,
fibered links, and open books in dimension $3$, 
see Etnyre \cite{Etnyre2006}
or Geiges \cite{Geiges2008}.

The first four sections of Ozbagci and Popescu-Pampu
\cite{OzbagciPopescu-Pampu2014} form an excellent historical
survey of plumbing and many of its generalizations.  
Starting from first principles, Bonahon and Siebenmann
\cite[Chapter 12]{BonahonSiebenmann1979/2010} give a 
careful treatment of---and calculus for---a particular 
case (called {strip-plumbing} below, 
\ref{rmks:2p-plumbing}\eqref{defs:iterated strip plumbing})
that explicitly allows non-orientable plumbands and is often 
suppressed in or excluded from such discussions.

\subsubsection{Braids and braided surfaces}\label{subsect:braids}
Let $n\ge 1$.  The \bydef{$n$--string braid group $B_n$} 
with identity $o^{(n)}$, \bydef{standard generators}
$\sigma_1,\dots,\sigma_{n-1}$, and 
\bydef{standard presentation}
\begin{equation*}
\left\langle \sigma_1,\dots,\sigma_{n-1}
	\left| 
	\begin{aligned}
	&\sigma_i\sigma_j\invof{\sigma_i}\invof{\sigma_j}
	&= o^{(n)},
	&\qua 1\le i<j-1\le n-1\\
	&\sigma_i\sigma_{i+1}\sigma_i
		\invof{\sigma_{i+1}}\invof{\sigma_{i}}\invof{\sigma_{i+1}}
	&=o^{(n)}, 
	&\qua 1\le i\le n-2
	\end{aligned}
	\right.
\right\rangle
\end{equation*}
is identified to the fundamental group of the
\bydef{configuration space}
\begin{align*}
E_n\isdefinedas 
	& \{\{w_1,\dots,w_n\}\sub\C\setsuchthat
	w_i\ne w_j, 1\le i<j\le n\}\\
\phantom{E_n}\sub & \{\{w_1,\dots,w_n\}\sub\C\} \cong 
\C^n/\symgroup{n}
\end{align*}
with respect to an arbitrary choice of base point 
$\boldsymbol{\omega}=\{w_1,\dots,w_n\}\in E_n$. 
A \bydef{positive band} in $B_n$ is any member of 
the conjugacy class of the standard generators; 
this conjugacy class is independent of
the choice of $\boldsymbol{\omega}$---in $E_n$ 
it is represented by any positively oriented meridian 
of the \bydef{discriminant locus} consisting of all
multisets $\{w_1,\dots,w_n\}\sub \C$ with 
$w_i=w_j$ for some $i\ne j$.  For $\beta, \gamma\in B_n$,
let $\SideSet{\gamma}{\beta}\isdefinedas 
\gamma\beta\invof{\gamma}$;
since any two standard generators $\sigma_i$, $\sigma_j$
are conjugate, every positive band has the form 
$\SideSet{\gamma}{\sigma_1}$ with $\gamma\in B_n$.
A braid $\beta\in B_n$ is \bydef{quasipositive} in 
case it belongs to the submonoid $Q_n\sub B_n$ generated 
by the positive bands (equivalently, normally generated by $\sigma_1$). 

The \bydef{closure} (or \bydef{closed braid}) of a braid 
$\beta$ is a smooth oriented link $(\close{\beta},S^3)$, 
unique up to ambient isotopy, defined as follows.  Let 
$\ell_\beta\from (S^1,1)\to (E_n,\boldsymbol{\omega})$ 
be a smooth based loop that represents $\beta\in B_n$. 
The \bydef{multigraph} 
$\gr{\ell_\beta}=\{(\mathrm{e}^{i\theta},w)
\in S^1\times\C\setsuchthat
w\in \ell_\beta(\mathrm{e}^{i\theta})\}$ is 
then a naturally oriented 
smooth compact $1$--submanifold of the open solid torus
$S^1\times\C$ such that $\pr_1\restr\gr{\ell_\beta}$ is a 
covering map.  Embed $S^1\times\C$ as the interior of one
solid torus of a genus--$1$ Heegaard splitting
of $S^3\sub \C^2$---say by the map
\begin{equation}
J\from S^1\times\C\to S^3\mapsuchthat 
(\mathrm{e}^{i\theta},w)\mapsto
\frac{(\mathrm{e}^{i\theta}\sqrt{1+|w|^2},w)}%
{\sqrt{1+2|w|^2}}\label{eqn:closed braid embedding}
\end{equation}
---and then define $\close\beta$ as the image 
$J(\gr{\ell_\beta})$.  
An oriented link $(L,S^3)$ is \bydef{quasipositive}
in case it is ambient isotopic to the closure $(\close\beta,S^3)$ 
of some quasipositive braid $\beta$.
A \bydef{quasipositive band representation} 
of a (necessarily quasipositive) braid 
$\beta\in B_n$ is a $k$--tuple $\brep{b}=(b(1),\dots,b(k))$ 
of positive bands in $B_n$ such that $\beta=\braidof{\brep{b}}
\isdefinedas b(1)\cdots b(k)$.  The calculus of 
band representations and \bydef{braided Seifert ribbons} 
in $D^4$ elaborated by Rudolph \cite{Rudolph1983b}, 
coupled with the equivalence
\cite{Rudolph1983,BoileauOrevkov2001} between 
$1$--dimensional transverse $\C$--links and 
quasipositive links (mentioned in Part \ref{part:introduction}), 
establishes a many-many correspondence between non-singular
$\C$--spans of $1$--dimensional transverse $\C$--links in $S^3$
and quasipositive band representations (on all numbers of
strings).  

\subsubsection{Annuli, strips, plumbing, and trees}%
\label{subsubsect:annuli, strips, and plumbing}
A \bydef{Seifert surface} is a compact oriented smooth
$2$--submanifold-with-boundary $S\sub S^3$ each component 
of which has non-empty boundary; $S$ is called a Seifert 
surface \bydef{for} (or \bydef{of}) the oriented link 
$(\Bd S,S^3)$, and $(\Bd S,S^3)$ is said to \bydef{have}
the Seifert surface $S$.

Let $\Lscr=(L,S^3)$ be a smooth classical link.
A \bydef{framing} of $\Lscr$ is a locally constant
function $f\from L\to \Z$; in case $\Lscr$ is a knot
(ie $L$ is connected), $f$ is identified with its only 
value.   An \bydef{annular surface $\AKn{\Lscr}{f}$} of
\bydef{type $(\Lscr,f)$} is a Seifert surface in 
$S^3$ consisting of pairwise disjoint annuli, each of
which contains exactly one component $K$ of $L$ as its
core $1$--sphere, and such that the linking number
in $S^3$ of the two boundary components of that
annulus is $-f(K)$ (in other words, the Seifert
matrix of that component is $[f(K)]$).  The ambient 
isotopy type of $\AKn{\Lscr}{f}$ is independent of 
the orientation of $L$.  The annular surface $\AKn{\Oscr}{-1}$ 
(where $\Oscr=(O,S^3)$ denotes a trivial knot) is often 
called a \bydef{positive Hopf band} (and 
its mirror image $\AKn{\Oscr}{+1}$ a \bydef{negative Hopf band}); 
to avoid possible (if unlikely) confusion with bands 
in braid groups, here I will call $\AKn{\Oscr}{\mp1}$ 
\bydef{Hopf annuli} instead (see \ref{subsect:qp Hopf links}
for some justification of the sobriquet ``Hopf''). 
 
Given a manifold $X$ (not necessarily oriented or orientable), 
let $\abs{X}$ denote the underlying unoriented manifold; 
given a link $\Lscr(L,M)$ with link-manifold $L$, let 
$\abs{\Lscr(L,M)}$ denote the unoriented link $(\abs{L},M)$
(so $M$ retains its orientation).  

\begin{definition}\label{def:strips}
Let $\Kscr=(K,S^3)$ be a classical knot, $t\in\Z$.  
A \bydef{strip of type $\Kscr$ with $t$ half-twists} 
is an unoriented $2$--submanifold-with-boundary
$\strip{\Kscr}{t}\sub S^3$ defined as follows:
\begin{inparaenum}[(1)]
\item\label{subdef:annular strips}
in case $t$ is even, $\strip{\Kscr}{t}=
\abs{\AKn{\Kscr}{-t/2}}$;
\item\label{subdef:Moebius strips}
in case $t$ is odd, $\strip{\Kscr}{t}$ is a smoothly
embedded M\"obius strip $S\sub S^3$ containing $K$ as its core
$1$--sphere, such if the $1$--sphere $\Bd S$ is oriented to 
be everywhere locally parallel (rather than anti-parallel)
to $K$, then the linking number in $S^3$ of $K$ and $\Bd S$
equals $t$.
\end{inparaenum}
Clearly up to ambient isotopies 
$\strip{\Kscr}{t}$ determines, and is determined by,
$\abs{\Kscr}$ and $t$.\done
\end{definition}

Recall that an arc $\alpha$ in a manifold with boundary $X$ is 
\bydef{proper} in case $\Bd\alpha=\alpha\cap\Bd X$.

\begin{definition}\label{def:M-sum}
Let $2p\ge 2$ be even.  Call    
a compact, smooth $2$--submanifold-with-boundary 
$F\sub S^3$, not necessarily oriented or orientable,
a \bydef{$2p$--gonal plumbing} of 
submanifolds-with-boundary $F_1, F_2\sub S^3$
\bydef{along $P$} 
in case there exists a smoothly embedded $2$--sphere 
$S^2\sub S^3$ bounding $3$--disks $D^3_1$, $D^3_2$ such that 
\begin{inparaenum}[(1)]
\item\label{subdef:plumbing sphere}
$F_i=F\cap D^3_i$ ($i=1,2$),
\item\label{subdef:plumbing patch}
$F\cap S^2=F_1\cap S^2=F_2\cap S^2$ is a $2$--disk $P$ 
such that $\Bd P$ consists of $2p$ arcs that are, alternately, 
proper arcs in $F_1$ and in $F_2$. 
$F_1$ and $F_2$ may be called the \bydef{plumbands}
of this plumbing; $P$ is its \bydef{plumbing patch}. 
\end{inparaenum}
\done
\end{definition}

\begin{figure}[ht!]
\centering
\includegraphics[width=\figwidth]{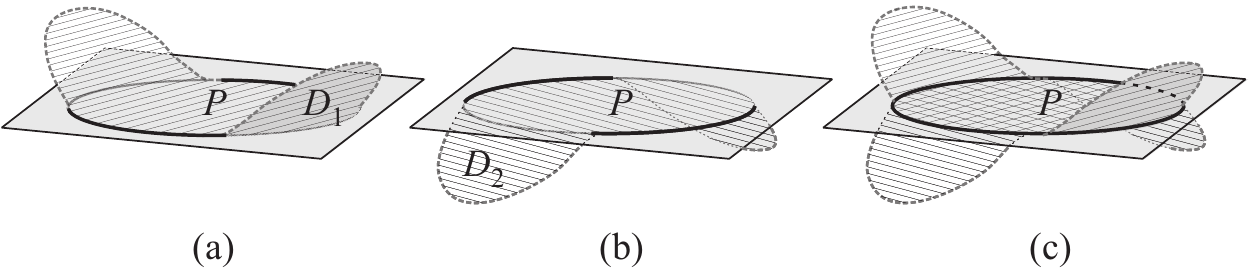} 
\vspace{-6pt}\caption{(a)~A $2$--disk $D_1\sub F_1\sub\R^2\times[0,\infty{[}$ intersects $\R^2\times\{0\}$ in the round $2$--disk~$P$.
(b)~A similar $2$--disk $D_2\sub F_2\sub\R^2\times{]}-\infty,0]$.
(c)~$D_1\cup D_2$ is a $2$--disk
on $F=F_1\plumb{P}F_2\sub\R^3$.
\label{fig:4-gonal M-sum}}
\vspace{-6pt}
\end{figure}

\begin{remarks}\label{rmks:2p-plumbing}
\begin{inparaenum}[(1)]
\item\label{rmk:connsum is M-sum}
Boundary-connected sum and $2$--gonal plumbing  
are equivalent.
\item\label{rmk:Stallings plumbing}
By \bydef{plumbing}, Stallings \cite{Stallings1978} refers
to a construction that, on its face, is a strict generalization
of $2p$--gonal plumbing; however, 
as observed in \cite[p.~260]{Rudolph1998}, ``it is easy to 
see that (up to ambient isotopy) every Stallings plumbing
is a'' $2p$--gonal plumbing ``of the same plumbands''.
\item\label{rmk:Gabai M-sum}
By his (now standard) coinage \bydef{Murasugi sum},
Gabai \cite{Gabai1983} refers exclusively to $2p$--gonal 
plumbing of Seifert surfaces.
\item\label{rmk:BS 2-plumbing}
For Bonahon and Siebenmann \cite{BonahonSiebenmann1979,
BonahonSiebenmann1979/2010}, the term ``plumbing'' refers
exclusively to $4$--gonal plumbing with unoriented 
(possibly non-orientable) plumbands $F_i$.
\item\label{rmk:no corners}
As pointed out in \cite[Remark 12.3]{BonahonSiebenmann1979/2010},
to stay in the differentiable category when plumbing, 
care ``can easily (and must)'' be taken to avoid creating 
corners (on $\Bd P$).  Such care is illustrated
in Figure \ref{fig:4-gonal M-sum}, where
$S^2=(\R^2\times\{0\})\cup\infty\sub\R^3\cup\infty=S^3$,
$P=\{(x_1,x_2,0)\in S^2\setsuchthat x_1^2+x_2^2\le 1\}$,
and liberal use of ``bump functions'' ensures 
that when the illustrated $2$--disks on $F_1$ 
and $F_2$ are identified along $P$, $\Bd F$
acquires no corners.
\item\label{rmk:stars and sides}
Although in the situation of \ref{def:M-sum} such notations 
as $F=F_1\plumb{P}F_2$ or---when $F$ is being constructed 
by plumbing, rather than displayed as already 
plumbed---$F=F_1\splumb{P_1}{P_2}F_2$---are often
useful, it is important to note that 
$F$ is typically not determined (even up to ambient 
isotopy) by the pairs $(F_i,P)$ or $(F_i,P_i)$
(or their ambient isotopy types): further (combinatorial) 
information
(such as \bydef{$n$-stars} and a distinction 
between the \bydef{sides} of $F_i$ near $P$), 
sufficient to specify an identification of 
$P_1$ and $P_2$ up to an appropriate equivalence, is 
required for disambiguation (with a few exceptions); 
see \cite[Remark 12.1]{BonahonSiebenmann1979/2010}
for unoriented plumbands and Rudolph \cite{Rudolph1998} 
for Murasugi sums.\done 
\end{inparaenum}
\end{remarks}

An especially useful case of plumbing in the 
sense of \cite{BonahonSiebenmann1979/2010} is 
\bydef{strip-plumbing}, where $F_2$ is an 
unknotted strip $\strip{\Oscr}{t}\sub S^3$ and the 
$4$--gonal plumbing patch $P_2\sub F_2$ is 
\bydef{core-transverse} in the sense that it is 
a relative regular neighborhood on $F_2$ of a 
\bydef{normal arc} (a proper arc that
intersects the core $1$--sphere $O\sub F_2$ in
a single point, transversely).  Iterating 
strip-plumbing produces several (overlapping) 
families of unoriented $2$--submanifolds-with-boundary
of $S^3$.

\begin{definitions}\label{defs:iterated strip plumbing}
Let $F_0\sub S^3$ be a $2$--disk. For $j=1,\dots,k$, 
let $F_{j-1}\splumb{Q_j}{P_j}\strip{\Oscr}{t_j}\defines
F_j$ be a strip-plumbing.  
\begin{inparaenum}[(1)]
\item\label{def:basket plumbing}  
If all $Q_i$ are contained in $F_0$ and all $t_i$ are
even, then $F_k$ is orientable, and with either orientation 
it is a \bydef{basket} as defined by Rudolph \cite{Rudolph2001a} 
and further studied by Furihata, Hirasawa and Kobayashi
\cite{FurihataHirasawaKobayashil2008},
Kim, Kwan, and Lee \cite{KimKwanLee2013}, etc. 
\item\label{def:Hopf plumbing} 
If all $\abs{t_i}$ equal $2$ (ie, all strips are Hopf bands), 
then again $F_k$ is orientable, and with either orientation it 
is a \bydef{Hopf-plumbed surface} as studied by Harer 
\cite{Harer1982}, Melvin and Morton \cite{MelvinMorton1986},
Rudolph \cite{Rudolph2001a}, Goodman \cite{Goodman2003}, 
etc.
\item\label{def:arborescent plumbing, traditional version}
Let $\Ts=(\Verts{\Ts},\Edges{\Ts})$ be a {planar tree},
$w\from\Verts{\Ts}\to\Z$ a {weighting} of $\Ts$. 
There seems to be no single standard notation or name for the 
unoriented, possibly non-orientable $2$-submanifold-with-boundary
of $S^3$ associated to $(\Ts,w)$ that has been described 
and constructed by various authors since (at least) 
Bonahon and Siebenmann \cite{BonahonSiebenmann1979};
here it will be denoted $\stripplumb{\Ts,w}$ and called the
\bydef{iterated strip-plumbed surface} of the weighted
graph (although the strip-plumbing in its construction 
seems usually to be conceptualized as simultaneous rather 
than iterated).  More precisely,
let $v_1,\dots,v_m$ be an enumeration of the vertices
of $\Ts$, and let $S_i\isdefinedas\strip{\Oscr}{w(v_i)}$.
For each edge $\{v_i,v_j\}$ of $\Ts$, with $i<j$, let
$P_{i,j}\sub S_i,P_{j,i}\sub S_j$ be core-transverse 
$4$--gonal plumbing patches, such that
\begin{inparaenum}[(a)]
\item\label{subdef:disjoint patches}
$P_{i,j}\cap P_{k,\ell}=\emptyset$ unless $(i,j)=(k,\ell)$
and
\item\label{subdef:cyclic order}
for each vertex $i$ the counter-clockwise cyclic order 
induced on $\{j\setsuchthat \{i,j\} \text{ is an edge of $\Ts$}\}$ 
by the hypothesized planar embedding of 
(the geometric realization of) the planar tree $\Ts$
is the same as the cyclic order in which the various 
plumbing patches $P_{i,j}$ and $P_{j,i}$ 
intersect the core $S^1$ of $S_i$.
\end{inparaenum}
Without loss of generality, the enumeration 
$v_1,\dots,v_m$ is such that $v_1,\dots,v_\ell$ are
vertices of a subtree of $\Ts$ for all $\ell=1,\dots,m$;
in this case, assumptions \eqref{subdef:disjoint patches}
and \eqref{subdef:cyclic order} suffice to construct
an iterated strip-plumbed $2$--manifold-with-boundary
\begin{equation}
\label{eqn:sp(T,w)}
(\cdots(S_1\splumb{P_{1,2}}%
{P_{2,1}}S_2)\cdots)\splumb{P_{q,m}}{P_{m,q}} S_m \sub S^3
\end{equation}
that typically is not unique up to ambient isotopy (see 
\ref{rmks:2p-plumbing}\eqref{rmk:stars and sides}).  However,
its boundary is, which excuses  the slight abuse of letting
$\stripplumb{\Ts,w}$ denote any of the surfaces \eqref{eqn:sp(T,w)}.
An \bydef{arborescent link} is an unoriented classical 
link of the form $(\Bd\stripplumb{\Ts,w},S^3)$.  For
$\Ts$ empty, let $\stripplumb{\Ts,w}=D^2$, so that
the unoriented trivial knot is arborescent.
\done
\end{inparaenum}
\end{definitions}

\subsubsection{Contact structures, fibered links, and open books}%
\label{subsubsect:fibered links, open books, and contact structures}
Let $M$ be a compact, oriented, smooth manifold of odd 
dimension $2n+1\ge 3$.  A \bydef{contact form} on $M$ 
is a $1$--form $\alpha$ on $M$ such that the $(2n+1)$--form 
$\alpha\wedge d\alpha\wedge\dotsm\wedge d\alpha$ 
(with $n$ factors $d\alpha$) is a volume form on $M$.  
A \bydef{contact structure} on $M$ is a field $\xi$ 
of $2n$--planes on $M$ of the form $\ker(\alpha)$ 
for some contact form $\alpha$.  Every strictly pseudoconvex 
$(2n+1)$--sphere $\Sigma\sub\C^{n+1}$ (in particular 
$S^{2n+1}=\SPH{2n+1}{\zero}{1}$) is equipped with a 
\bydef{canonical contact structure}, namely, the field 
of oriented $2n$--planes underlying the field of complex 
$n$--planes tangent to $\Sigma$.  The link-manifold of 
any transverse $\C$--link (of dimension greater than or
equal to $3$) has a similarly defined canonical contact
structure.

An $n$--dimensional
smooth link $\Lscr=(L,M)$ is \bydef{Legendrian} for 
$\xi$ in case the tangent $n$--plane to $L$ at each 
of its points lies in the contact $2n$--plane of $M$
at that point.  If $(L,M)$ is Legendrian and $M$ is 
Riemannian, then the field of $n$--planes on $L$ 
complementary in the field of contact $2n$--planes
to the tangent bundle of $L$ is a natural normal 
$n$--plane field on $L$, independent (up to isotopy)
of the metric on $M$.

Let $2n+1=3$. A contact structure $\xi$ on $M$ 
is called \bydef{overtwisted} in case there exists a
Legendrian knot $\Kscr=(K,M)$ such that $K=\Bd D$
where $D\sub M$ is a smoothly embedded $2$--disk 
such that the restrictions to $K$ of $\xi$ and the 
tangent bundle of $D$ are homotopic.  A contact structure
is called \bydef{tight} in case it is not overtwisted.  

\begin{theorem}[Bennequin \cite{Bennequin1983}]
\label{thm:Bennequin's theorem}
The canonical contact structure $\xi_0$ on $S^3$ 
is tight.\qed
\end{theorem}

\begin{remark}\label{rmk:Stein fillable}
It follows from a (much) more general theorem of
Eliashberg and Gromov \cite{EliashbergGromov1991}
that the canonical contact structure on the link-manifold
$\lman{f}{\Sigma}$ of a $3$--dimensional transverse 
$\C$--link $\Lscr(f,\Sigma)$ is tight (because, up to 
replacing $f$ by $f+\epsilon$, the $\C$--span
$\Cspan{f}{\Delta}$ is non-singular and thus a
\bydef{Stein filling} of $\lman{f}{\Sigma}$; see
Gompf \cite{Gompf1998}). \done
\end{remark}

\begin{theorem}[Eliashberg \cite{Eliashberg1992}]
\label{thm:Eliashberg's theorem}
Overtwisted contact structures are isotopic iff they are
homotopic as plane fields.  Every homotopy class of
plane fields on $S^3$ contains overtwisted contact 
structures; only the class of $\xi_0$ contains a tight
contact structure.\qed
\end{theorem}

If a classical link $\Lscr=(L,S^3)$ is Legendrian for $\xi_0$,
then $\Lscr$ is naturally framed by assigning to each component
$K$ of $L$ the linking number in $S^3$ of $K$ with $K^{+}$ 
obtained by pushing $K$ a small distance along its natural
normal line field.  
Two standard facts are that every smooth classical link
$\Lscr$ is ambient isotopic to various Legendrian links for 
$\xi_0$, and that if $\Lscr=\Kscr$ is a knot then there is a 
finite upper bound---called the 
\bydef{maximal Thurston--Bennequin number} of $\Kscr$,
and denoted $\TB(\Kscr)$---for the self-linking numbers
of Legendrian knots smoothly isotopic to $\Kscr$.

Let $M$ be a smooth, oriented $m$--manifold of dimension
$m\ge 2$.  An $(m-2)$--dimensional link $\Lscr=(L,M)$ is
\bydef{fibered} in case there is a smooth fibration
$\phi\from M\setminus L\to S^1$ such that each fiber
$\phi^{-1}(\mathrm{e}^{i\theta})=\Int F_\theta$ for a smooth, 
compact $(m-1)$--dimensional submanifold-with-boundary
$F_\theta\sub M$ with $\Bd F_\theta=L$; such an 
$F_\theta$ (for any
$L$ and $\phi$) is a \bydef{fiber manifold in $M$}.
The mirror image $\Mir\Lscr\isdefinedas(L,-M)$ of a 
link $\Lscr=(L,M)$ is fibered iff $\Lscr$ is;
a connected sum $\Lscr_1\connsum\Lscr_2\isdefinedas
(L_1\connsum L_2,M_1\connsum M_2)$ of links
$\Lscr_i=(L_i,M_i)$ is fibered iff $\Lscr_1$ and
$\Lscr_2$ are, and similarly for boundary-connected
sums of fiber manifolds.

An \bydef{open book} on $M$ is a smooth map 
$\book{b}\from M \to\C$ such that $0\in\C$ is a regular 
value of $\book{b}$ and 
$(\book{b}/|\book{b}|)\restr \invof{\book{b}}\of{\C\setminus 0}
\from \invof{\book{b}}\of{\C\setminus 0}\to S^1$ is a fibration. 
A \bydef{page} of $\book{b}$ is any one of the smooth, oriented
$(m-1)$--manifolds-with-boundary
$F_\theta\of{\book{b}}\isdefinedas
\invof{\book{b}}\of{\{r\mathrm{e}^{i\theta}\setsuchthat r\ge0\}}$
for $\mathrm{e}^{i\theta}\in S^1$, with non-empty boundary 
$L\of{\book{b}}\isdefinedas \invof{\book{b}}\of0$.
The oriented link $\Lscr(\book{b})\isdefinedas (L(\book{b}),M)$
is the \bydef{binding} of $\book{b}$.
Open books on $M$ are handy rigidifications
of fibered links in $M$; indeed, $\Lscr(\book{b})$ is a 
fibered link in $M$, every non-empty fibered link in $M$ 
is $\Lscr(\book{b})$ for various open books on $M$ (all
equivalent in an appropriate sense), and every 
fiber manifold in $M$ with non-empty boundary 
is a page of some open book on 
$M$ (again, essentially unique).
For odd $m\ge 3$, an open book on (or fibered link in) $M$ 
is called \bydef{simple} in case its page (or fiber 
$(m-1)$--manifold) has the homotopy type of a bouquet 
of $(m-1)/2$--spheres; this is always so for $m=3$.  
For $m=3$ and $M=S^3$, 
Neuwirth \cite{Neuwirth1961,Neuwirth1963} and Stallings 
\cite{Stallings1962} showed independently that 
a Seifert surface $F$ is a \bydef{fiber surface}
(that is, a fiber manifold of dimension $2$) iff the 
normal push-off $F\to S^3\setminus F$ induces a homotopy 
equivalence.  It follows that $F$ is connected 
($H_0(F;\Z)$ is Alexander dual to $H_2(S^3\setminus F;\Z)$), 
and thus that $\AKn{\Lscr}{f}$ is fibered
iff $\Lscr=\Oscr$ and $f=\pm 1$.

\begin{remark}\label{rmk:analogy between qp and fibered}
As noted in the abstract of \cite{Rudolph1998}, there 
is an analogy between fiber surfaces in $S^3$ and
quasipositive Seifert surfaces in $S^3$:
``a Seifert surface $S\sub S^3 = \Bd{D^4}$ is a fiber 
surface if a push-off $S\to S^3\setminus S$ induces 
a homotopy equivalence; roughly, $S$ is quasipositive if
pushing $\Int S$ into $\Int D^4\sub \C^2$ produces a 
piece of complex plane curve.''  A glimpse
of this analogy led me to call the first of my  
series of papers \cite{Rudolph1983c,
Rudolph1984,
Rudolph1992b,
Rudolph1992a,
Rudolph1998} ``Constructions of quasipositive knots and 
links'' in homage to Stallings's paper ``Constructions
of fibred knots and links'' \cite{Stallings1978}.
Several of the following constructions can be taken 
as evidence that this analogy is not completely illusory.
\done
\end{remark}

\subsection{Construction: quasipositive satellites 
(new) \label{subsect:qp cable links}}
The following construction in classical knot theory is due to 
Schubert \cite{Schubert1953}.  

\begin{definition}\label{def:satellization}
Let $\Kscr=(K,S_0^3)$ be a knot, $\Lscr=(L,S_1^3)$ a link,
and $\Oscr=(O,S_1^3)$ a trivial knot such that
	\begin{inparaenum}[(a)]
	\item\label{subrmk:pattern siting a}
	$L$ is contained in the interior $\Int N(O)$ of a
	regular neighborhood $N(O)$ of $O$ in $S^3$, and
	\item\label{subrmk:pattern siting b}
	no $2$--sphere in $N(O)$ separates any connected
	component of $L$ from $\Bd N(O)$.
	\end{inparaenum}
Let $E(K)\isdefinedas S_0^3\setminus\Int N(K)$ be
the \bydef{exterior} of $K$, let
$h\from\Bd E(K)\to\Bd N(O)$ be a faithful
diffeomorphism (ie, it carries a standard 
meridian-longitude pair on $\Bd E(K)$ to a
standard meridian-longitude pair on $\Bd N(O)$),
and let $S^3$ be the (suitably smoothed) identification 
space $(E(K)\cup N(O))/{\sim}$, where the non-trivial
equivalence classes of the equivalence relation 
${\sim}$ are the pairs $\{x,h(x)\}$ with $x\in\Bd E(K)$.
Let $K\{L\}$ denote $L\sub N(O)\sub S^3$.  In this 
situation, the link $\Kscr\{\Lscr\}\isdefinedas(K\{L\},S^3)$
is the \bydef{satellite} of $\Kscr$ with \bydef{pattern}
$\Lscr$; $\Kscr$ is a \bydef{companion} of $\Kscr\{\Lscr\}$.
\done
\end{definition}

Stallings gave a natural condition under which a satellite 
with fibered companion and fibered pattern is itself fibered. 

\begin{theorem}[Stallings \cite{Stallings1978}]%
\label{thm:Stallings's satellization theorem}
If $\Kscr$ and $\Lscr$ are fibered links
in $S_0^3$ and $S_1^3$ respectively, and if, 
further, there exist an integer $d\ne 0$ and open books 
$\book{p}\from S_1^3\to D^2$ for $\Lscr$ and
$\book{o}\from S_1^3\to D^2$ for $O$ 
with $\book{p}\restr E(O)$ equal to
$(\book{o}\restr E(O))^d$ up to the identification above,
then $\Kscr\{\Lscr\}$ is a fibered link in $S^3$.
\qed 
\end{theorem}

The analogy mentioned in 
Remark~\ref{rmk:analogy between qp and fibered} suggests
that there should be a similarly broad result for a satellite
with quasipositive companion and quasipositive pattern, 
presumably subject to some further coherence condition 
like that in Theorem \ref{thm:Stallings's satellization theorem}. 
Lacking sufficient space, time, and insight either to 
find such a broad theorem or come up with convincing 
reasons none should exist, here I prove only a single
narrow result to be used later.

For $n\ge 1$, denote by $\Oscr^{(n)}\isdefinedas(O^{(n)},S^3)$ 
the closed braid of the identity $o^{(n)}\in B_n$, 
embedded---as in subsection \ref{subsect:braids}, equation
\eqref{eqn:closed braid embedding}---in 
$J(S^1\times\C)=\Int N(O^{1})\sub S^3=S_1^3$.
The \bydef{untwisted $n$--strand cable} of a knot
$\Kscr=(K,S_0^3)$ is the satellite 
$\Kscr\{n,0\}\isdefinedas\Kscr\{\Oscr^{(n)}\}$.

\begin{proposition}\label{prop:untwisted cables on qp}
If $\Kscr$ is quasipositive, then for all $n\ge 1$,
$\Kscr\{n,0\}$ is quasipositive.
\end{proposition}

\begin{proof}
Realize the quasipositive knot $\Kscr$ as a transverse 
$\C$--link---say $\Kscr=\Clink{f}{S^3}$---with 
non-singular $\C$--span $\Cspan{f}{D^4}$.
For all sufficiently small $\epsilon\ne 0$, 
$(V(f^n-\epsilon^n)\cap S^3,S^3)$ is then a transverse 
$\C$--link with $n$ components 
$(V(f-\epsilon\mathrm{e}^{2k\pi i/n})\cap S^3,S^3)$, 
each a transverse $\C$--link
in its own right, and such that its $\C$--span
$S_k\isdefinedas
(V(f-\epsilon\mathrm{e}^{2k\pi/n})\cap D^4,D^4)$
is non-singular.  For $1\le k<\ell\le n$, clearly
$S_k\cap S_\ell=\emptyset$, so the linking number of
$\Bd S_k$ and $\Bd S_\ell$ in $S^3$ is $0$; it follows
that $\Clink{f^n-\epsilon^n}{S^3}$ is, up to ambient
isotopy, $\Kscr\{n,0\}$. 
\end{proof}
 
\begin{remarks}\label{rmks:qp braided companion of qp braid}
\begin{inparaenum}[(1)]
\item\label{rmk:qp braided cables--algebraic}
For another proof of \ref{prop:untwisted cables on qp}, 
let $\Kscr$ be the closed braid  
$\close{\braidof{\brep{b}}}$ of a quasipositive  
band representation $\brep{b}$ in $B_p$; fairly obvious 
algebraic manipulations (motivated by geometry) 
generate a quasipositive band representation
$\brep{b}\{n,0\}$ in $B_{np}$ with closure $\Kscr\{n,0\}$.
\item\label{rmk:qp braided companion--algebraic}
The proof just sketched readily generalizes to show that 
$\Kscr\{\Lscr\}$ is quasipositive in case both
the companion $\Kscr$ and the pattern $\Lscr$ are 
quasipositive \emph{and in addition} $L$ sits inside
$N(O)$ as a quasipositive closed braid.  Certainly,
this last hypothesis \emph{is} a ``coherence condition like that 
in Theorem \ref{thm:Stallings's satellization theorem}'', but it
seems much too strong to be optimal (and is much stronger than
Stallings's condition).
\item\label{rmk:qp braided companion--analytic}
In the situation of \eqref{rmk:qp braided companion--algebraic},
if also the quasipositive companion $\Kscr$ is a slice 
knot, then an analytic proof that $\Kscr\{\Lscr\}$ is
quasipositive can be cobbled together along the lines of 
the (first) proof of \ref{prop:untwisted cables on qp}
by using techniques applied (in a much more delicate
context) by Baader, Kutzschebauch, and Wold \cite{Baaderetal2010}.  
\end{inparaenum}\done
\end{remarks}

\subsection{Construction: strongly quasipositive links%
\label{subsect:strongly qp links}}
The monoid $Q_n$ contains a distinguished
finite subset  
\begin{equation*}
\{\sigma_{i,j}\isdefinedas
\SideSet{\sigma_i\dotsm\sigma_{j-2}}{\sigma_{j-1}}
\setsuchthat 1\le i\le j\le n-1\}
\end{equation*} 
of positive bands called \bydef{embedded bands} 
(in $B_n$) by Rudolph \cite{Rudolph1983b} and later, 
a bit confusingly, simply ``band generators'' (of $B_n$) 
by Birman, Ko, and Lee \cite{BirmanKoLee1998}.
The calculus of 
band representations and Seifert ribbons in $D^4$ mentioned in
\ref{subsect:braids} has a variant (expounded, like it,
in \cite{Rudolph1983b}, and elaborated in various later
papers by Rudolph \cite{Rudolph1992b,Rudolph1992a,
Rudolph1998,Rudolph2001a}, Baader and Ishikawa
\cite{BaaderIshikawa2009,BaaderIshikawa2011}, etc)
by which \bydef{quasipositive embedded band representations}
$\brep{b}$
and algebraic/combinatorial operations thereon 
correspond to \bydef{quasipositive braided Seifert surfaces} 
$\brsurf{\brep{b}}$ in $S^3$ and geometric/topological 
operations thereon.
A Seifert surface is called \bydef{quasipositive}
in case it is ambient isotopic to a quasipositive braided
Seifert surface $\brsurf{\brep{b}}$. 

Given a compact orientable $2$--manifold-with-boundary
$M$, call a closed subset $N\sub\Int{M}$ \bydef{full}
on $M$ in case no component of $M\setminus N$ is contractible.

\begin{proposition}[Rudolph \cite{Rudolph1992b}]%
\label{prop:full on qp Ss} 
A Seifert surface is quasipositive iff it is a full 
subsurface of some quasipositive fiber surface.
In particular, a full subsurface of a quasipositive
Seifert surface is quasipositive.\qed
\end{proposition}

A link is \bydef{strongly quasipositive} in case it has
a quasipositive Seifert surface.   Many interesting 
quasipositive links are strongly quasipositive,
including the classes of examples described next.

\subsubsection{Strongly quasipositive annuli}
\label{subsubsect:qp annuli}

\begin{theorem}[Rudolph \cite{Rudolph1992a,Rudolph1992,Rudolph1995}]
\label{thm:qp annuli}
\begin{inparaenum}[\upshape(1)]
\item\label{subthm:qp knotted annuli and TB}
If the smooth classical knot $\Kscr$ is non-trivial, then
the following are equivalent:
	\begin{inparaenum}[\upshape(a)]
	\item
	the annular Seifert surface $\AKn{\Kscr}{n}$ is quasipositive;
	\item
	the oriented link $(\Bd\AKn{\Kscr}{n},S^3)$ is
	strongly quasipositive;
	\item
	$n\le\TB(\Kscr)$.
	\end{inparaenum}
\item\label{subthm:qp unknotted annuli and TB}
The oriented link $(\Bd\AKn{\Oscr}{n},S^3)$
is strongly quasi\-positive iff $n\le 0$; the annular 
surface $\AKn{\Oscr}{n}$ is quasipositive
iff $n\le -1=\TB(\Oscr)$.\qed
\end{inparaenum}
\end{theorem}

\subsubsection{Strongly quasipositive Murasugi sums%
\label{subsubsect:qp Murasugi sums}}

\begin{theorem}[Rudolph \cite{Rudolph1998}%
\label{thm:qp Murasugi sum theorem}]
A Murasugi sum of Seifert surfaces $F_1$ and $F_2$ is 
quasipositive iff the summands $F_1$ and $F_2$ are 
quasipositive.\qed
\end{theorem}

\begin{remarks}\label{rmks:qp Murasugi sum}
\begin{inparaenum}[(1)]
\item\label{rmk:qp Hopf-plumbed fiber surfaces}
Evidently $\strip{\Oscr}{2}=\abs{\AKn{\Oscr}{-1}}$, so 
Theorem \ref{thm:qp annuli}%
\eqref{subthm:qp unknotted annuli and TB} 
and Theorem \ref{thm:qp Murasugi sum theorem} imply 
that if each plumband of a Hopf-plumbed Seifert surface $F$ 
(as in \ref{defs:iterated strip plumbing}\eqref{def:Hopf plumbing})
is $\strip{\Oscr}{2}$, then $(\Bd F,S^3)$ is
strongly quasipositive.
\item\label{rmk:Gabai's credo for qp}
Theorem \ref{thm:qp Murasugi sum theorem} is analogous to
Gabai's theorem \cite{Gabai1985} that a Murasugi
sum of Seifert surfaces is a fiber surface iff the 
plumbands are fiber surfaces, and may be taken as 
further evidence (along different geometric lines from 
those followed in \cite{Gabai1983,Gabai1985} and later 
work by Gabai and others) for what Ozbagci and Popescu-Pampu 
\cite{OzbagciPopescu-Pampu2014} call \bydef{Gabai's credo}:
``the Murasugi sum is a natural geometric operation''.
\end{inparaenum}
\done
\end{remarks}

\subsubsection{Positive links\label{subsubsect:positive links}}
Given a classical oriented link diagram $\diagram$,
let $\SeifAlg{\diagram}$ denote the Seifert surface
(unique up to ambient isotopy) produced by  
\bydef{Seifert's algorithm} (Seifert, \cite{Seifert1935})
applied to $\diagram$, so $\Lscr(\diagram)\isdefinedas
(\Bd\SeifAlg{\diagram},S^3)$ 
is the oriented link (unique up to ambient isotopy) 
determined by $\diagram$.  A diagram is \bydef{positive}
in case every crossing is positive; a link is 
\bydef{positive} in case it has some positive $\diagram$.
Positivity (of links and diagrams) is preserved by simultaneous
reversal of all orientations---in particular, for \textit{knots}
it is independent of orientation.

\begin{theorem}[Nakamura \cite{Nakamura1998,Nakamura2000},
Rudolph \cite{Rudolph1999}]\label{thm:positive implies stqp}
If $\diagram$ is positive, then $\SeifAlg{\diagram}$
is a quasipositive Seifert surface.  In particular,
a positive link is strongly quasipositive. 
\qed
\end{theorem}

\subsubsection{Strongly quasipositive satellites (new)}
\label{subsect:strongly qp satellites}

Two results stated for quasipositive knots and links 
in section \ref{subsect:qp cable links} remain true in the 
strongly quasipositive case.  In the first, a variation on \ref{prop:untwisted cables on qp}, the proofs differ a bit.

\begin{proposition}\label{prop:untwisted cables on strong qp}
If $\Kscr$ is strongly quasipositive, then for 
all $n\ge 1$, $\Kscr\{n,0\}$ is strongly quasipositive.
\end{proposition}

\begin{proof}
Let $S$ be a quasipositive Seifert surface with 
$\Kscr=(\Bd S,S^3)$.  Let $c\from S\times[1,n]\to S^3$ 
be an embedding onto a one-sided collar of $S=c(S\times\{1\})$; 
then $c(S\times\{1,\dots,n\})$ is a quasipositive Seifert surface, 
and its boundary is clearly $\Kscr\{n,0\}$.
\end{proof}

The second is a variation on 
\ref{rmks:qp braided companion of qp braid}%
\eqref{rmk:qp braided companion--algebraic}; in this case,
the sketched proof of the original applies equally well
to the variation.

\begin{proposition}\label{prop:strongly qp satellites}
If both the companion $\Kscr$ and the pattern $\Lscr$
are strongly quasipositive, and if in addition $L$ sits 
inside $N(O)$ as a strongly quasipositive closed braid, 
then the satellite $\Kscr\{\Lscr\}$ is strongly quasipositive.
\qed
\end{proposition}

\subsection{Construction: partially reoriented Hopf links 
(new details)\label{subsect:qp Hopf links}}
The \bydef{partially reoriented positive Hopf links} 
$\Hopf{+}{p}{q}$ and their mirror images the
\bydef{partially reoriented negative Hopf links}
$\Hopf{-}{p}{q}\isdefinedas \Mir\Hopf{+}{p}{q}$ are 
defined using the \bydef{positive Hopf fibration} 
$h_{+}\from S^3\to\CP1\mapsuchthat(z_0,z_1)\mapsto(z_0:z_1)$
and its mirror image the \bydef{negative Hopf fibration}
$h_{-}\mapsuchthat(z_0,z_1)\mapsto(\conj{z_0}:\conj{z_1})$.
The usual orientations of $S^3\sub\C^2$ and $\CP1$ naturally
orient the fibers of $h_\pm$. 
For $0\ne p\ge q\ge 0$, denote by $H_{\pm}(p,q)$  
the union of (any) $p+q$ fibers of $h_{\pm}$, 
$p$ with the natural orientation and $q$ with its opposite;
let $\Hopf{\pm}{p}{q}\isdefinedas (H_{\pm}(p,q),S^3)$.  Note
that $\Hopf{+}{1}{0}$ and $\Hopf{-}{1}{0}$ are ambient isotopic
(they are trivial knots), as are $\Hopf{\pm}{2}{0}$ and
$\Hopf{\mp}{1}{1}$; with those exceptions, $\Hopf{\pm}{p}{q}$
is determined up to ambient isotopy by $(\pm,p,q)$.
Let $\nabla_n\isdefinedas ((\sigma_1\sigma_2\dotsm\sigma_{n-1})
(\sigma_1\sigma_2\dotsm\sigma_{n-2})\dotsm
(\sigma_1\sigma_2)\sigma_1)^2\defines\Delta\in B_n$.
It is standard that the closure of $\nabla_n$ is 
$\Hopf{\pm}{n}{0}$.  For $1\le i<j\le n$, inject
$B_{j-i+1}$ into $B_n$ by $\iota_{i,j}$ with
$\iota_{i,j}(\sigma_k)=\sigma_{k+i-1}$, $k=1,\dots,j-i$;
let $\nabla_{i,j}\isdefinedas \iota_{i,j}(\nabla_{j-i+1})$.
Figure \ref{fig:Hopf links}(a)--(b)
show that
\begin{align*}
\nabla_{p+q}
&=\nabla_{1,\;p}\nabla_{p+1,\;p+q}%
(\sigma_p\sigma_{p+1}\dotsm\sigma_{p+q-1})%
(\sigma_{p-1}\sigma_p\dotsm\sigma_{p+q-2})
\dotsm (\sigma_1\sigma_2\dotsm\sigma_q)\\
&\quad(\sigma_q\sigma_{q+1}\dotsm\sigma_{q+p-1})
(\sigma_{q-1}\sigma_q\dotsm\sigma_{q+p-2})
\dotsm(\sigma_1\sigma_2\dotsm\sigma_p)\\
&=\nabla_{1,\;p-q}\nabla_{p-q+1,\;p+q}
(\sigma_{p-q}\sigma_{p-q+1}\dotsm\sigma_{p-1})
(\sigma_{p-q-1}\sigma_{p-q}\dotsm\sigma_{p-2})
 \dotsm\\
&\quad (\sigma_1\sigma_2\dotsm\sigma_{p-q})
(\sigma_{p-q}\sigma_{p-q+1}\dotsm\sigma_{p+q-1})
(\sigma_{p-q-1}\sigma_q \sigma_{p+q-2})
\dotsm(\sigma_1\sigma_2\dotsm\sigma_{2q})
\end{align*}
and thus both have closure $\Hopf{+}{p+q}{0}$.
\begin{figure}[ht!]
\begin{center}
\includegraphics[width=\figwidth]{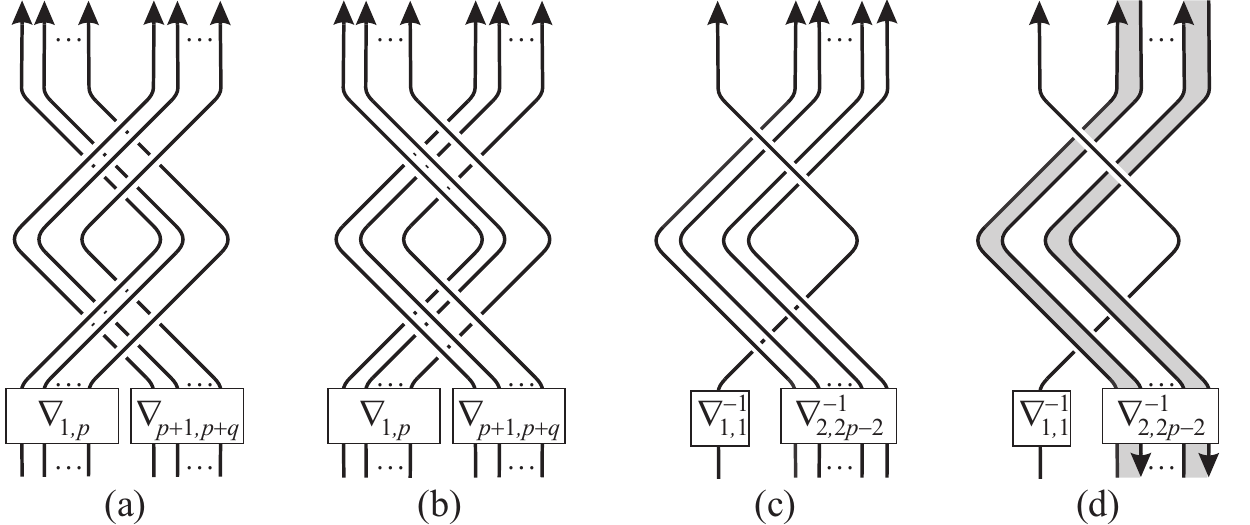} 
\caption{(a), (b), (c)~Braid diagrams with closure 
$\Hopf{+}{p+q}{0}$, $\Hopf{+}{p}{q}$, $\Hopf{-}{2p-1}{0}$
respectively.
(d)~A partially reoriented braid diagram with closure
$\Hopf{-}{p}{p-1}$; the shading indicates $p-1$ linked 
and twisted annuli that are components of a ribbon-immersed 
surface in $S^3$ bounded by $\Hopf{-}{p}{p-1}$ (the remaining
component is a disk, not indicated).
\label{fig:Hopf links}}
\end{center}
\end{figure}%
(Since $\iota_{1,\;r}(B_r)$ and $\iota_{r+1,\;p+q}(B_{p+q-r})$ 
commute with each other for any $r=1,\dots,p+q$, the 
detailed placement of crossings inside 
the boxes at the bottoms of the diagrams is irrelevant; 
a similar observation applies to the tops of the diagrams.) 
The braid diagram in Figure \ref{fig:Hopf links}(c) is 
derived from that in Figure \ref{fig:Hopf links}(a) by 
simultaneously reversing the orientations of the 
rightmost $q$ strings and turning those strings, 
so grouped, by (approximately) a half-turn 
around the horizontal axis; its closure is evidently 
$\Hopf{+}{p}{q}$.  Although Figure \ref{fig:Hopf links}(d)%
---derived from Figure \ref{fig:Hopf links}(b) by reversing
the orientation of alternate ones of the last $2q$ strings%
---is not a braid diagram for $q>0$, it has an obvious 
``closure'' that is, again evidently, $\Hopf{+}{p}{q}$.
(The shading in Figure \ref{fig:Hopf links}(d) is for future 
reference.)

\begin{lemma}\label{lemma:qp Hopf links}
Let $p\ge q\ge 0$.
\begin{inparaenum}[\upshape(1)]
\item
$\Hopf{+}{p}{q}$ is quasipositive iff $p\ge 1$ and $q=0$.
\item 
$\Hopf{-}{p}{q}$ is quasipositive iff either
$q=p>0$ or $q=p-1$.
\end{inparaenum}\qed
\end{lemma}

In particular, a Seifert surface diffeomorphic to $D^2$ 
is a fiber surface bounded by the trivial fibered knot
$\Oscr=\Hopf{\pm}{1}{0}$; and a Seifert surface 
$\AKn{\Kscr}{n}$ diffeomorphic to an annulus (see
\fullref{subsubsect:qp annuli} for the notation) is 
a fiber surface iff it is a 
\bydef{${\pm}$-ive Hopf band}
$\AKn{\Oscr}{\mp 1}$ bounded by the fibered 
\bydef{${\pm}$-ive Hopf link} $\Hopf{\pm}{2}{0}=\Hopf{\mp}{1}{1}$.

\subsection{Construction: quasipositive fibered links 
(new details)\label{subsect:qp fibered links}}
\begin{lemma}\label{lemma:fibered Hopf links}
For $p>1$, $\Hopf{\pm}{p}{q}$ is a fibered link iff $p>q$.
\end{lemma}

\begin{proof}
Calculations in the style of Rudolph \cite{Rudolph1987} 
show that if $p>q$ then the real-polynomial mapping
\begin{equation*}F_{p,q}\from\C^2\to\C
\mapsuchthat(z_0,z_1)\mapsto
(z_0^{p}+z_1^{p})(\conj{z_0}^q+2\conj{z_1}^q)
\end{equation*}
has an isolated critical point at $(0,0)$, and in fact
that $F_{p,q}\restr S^3$ is an open book with binding   
$\Hopf{+}{p}{q}$; the result for $\Hopf{-}{p}{q}$ follows
by taking mirror images. On the other hand, if $p>1$ 
then $\Hopf{\pm}{p}{p}$ is not fibered (it has 
a disconnected Seifert surface, so 
$S^3\setminus H_{\pm}(p,p)$ has non-trivial
second homology and cannot be homotopy equivalent 
to a bouquet of $1$--spheres).  
Alternatively, note that a partially reoriented Hopf link is \bydef{solvable} in the sense of Eisenbud and Neumann
\cite{EisenbudNeumann1985} and then apply the characterization 
of fibered solvable links derived in \cite{EisenbudNeumann1985}
using the calculus of splice diagrams.
\end{proof}

In combination with \ref{lemma:qp Hopf links},
\ref{lemma:fibered Hopf links} yields the following.
\begin{corollary}\label{cor:qp fibered Hpq}
$\Hopf{p}{q}$ is both quasipositive and fibered iff
it is the trivial knot, the positive Hopf link,
$\Hopf{+}{p}{0}$, or $\Hopf{-}{p}{p-1}$.
\qed
\end{corollary}


Let $\Lscr$ be a simple fibered link in $S^{2n+1}$, 
$\book{p}$ an open book with $\Lscr=\Lscr_\book{p}$.
The \bydef{Milnor number} $\mu(\Lscr)$ is now usually 
defined as the middle Betti number of the fiber 
$2n$--manifold of $\Lscr$, making the following 
properties evident.

\begin{proposition}\label{prop:Milnor number properties}
\begin{inparaenum}[\upshape(1)]
\item\label{subdef:mu non-neg}
$\mu(\Lscr)\ge 0$.
\item\label{subdef:mu achiral}
$\mu(\Mir\Lscr)=\mu(\Lscr)$.
\item\label{subdef:mu(O)=0}
$\mu(\Lscr)=0$ iff the fiber $2n$--manifold is 
contractible; in particular, for $n=1$, $\mu(\Lscr)=0$
iff $\Lscr=\Oscr$ is a trivial knot.
\item\label{subdef:mu is additive}
If $\Lscr_1$ and $\Lscr_2$ are fibered links,
then $\mu(\Lscr_1\connsum\Lscr_2)=
\mu(\Lscr_1)+\mu(\Lscr_2)$.
\end{inparaenum}\qed
\end{proposition}  
Originally, however, $\mu$ was defined by Milnor 
\cite{Milnor1968} (in his context of links 
$\Clinksing{\zz}{f}$ of isolated singular points 
of complex hypersurfaces $V(f)\sub\C^{n+1}$, $n\ge 1$,
where $\book{p}=f\restr\SPH{2n+1}{\zz}{\epsilon}$)  
as the degree of a map $S^{2n+1}\to S^{2n+1}$ naturally
associated to $\book{p}$, while the Betti number 
characterization was a theorem to be proven
and \ref{prop:Milnor number properties}%
\eqref{subdef:mu non-neg}--\eqref{subdef:mu(O)=0}
were its corollaries (in Milnor's context, 
\eqref{subdef:mu is additive} arises only trivially). 

For $n=1$, Rudolph \cite{Rudolph1987} adapted Milnor's 
original approach to $\mu$ to define for every fibered 
classical link $\Lscr$---in terms of any open book 
$\book{p}$ with $\Lscr=\Lscr_\book{p}$---a pair 
$(L_\Lscr,R_\Lscr)$ of maps $S^3\to S^2$ naturally 
associated to $\book{p}$.  In \cite{Rudolph1987}, the 
pair $(\lambda(\Lscr),\rho(\Lscr))$ of Hopf invariants
of $(L_\Lscr,R_\Lscr)$ was called the \bydef{enhanced 
Milnor number} of $\Lscr$, and shown to have the 
following properties.  

\begin{proposition}\label{prop:lambda facts}
\begin{inparaenum}[\upshape(1)]
\item\label{subprop:lambda plus rho}
$\lambda(\Lscr)+\rho(\Lscr)=\mu(\Lscr)$.
\item\label{subprop:lambda mirrored}
$\rho(\Lscr)=\lambda(\Mir\Lscr)$.
\item\label{subprop:lambda 0 on links of singularities}
$\lambda(\Clinksing{\zz}{f})=0$ if $f\from\C^2\to\C$ is
a holomorphic function with an isolated critical point (or
regular point) at $\zz\in\C^2$.
\item\label{subprop:lambda additivity}
$\lambda$ is additive over connected sum: $\lambda(\Lscr_1\connsum\Lscr_2)=
\lambda(\Lscr_1)+\lambda(\Lscr_2)$.
\end{inparaenum}
\qed
\end{proposition}

Neumann and Rudolph 
\cite{NeumannRudolph1987,
NeumannRudolph1988,
NeumannRudolph1990} named $\lambda(\Lscr)$ (and its analogue
for fibered links of higher odd dimension, an element of 
$\Z/2\Z$ rather than $\Z$) the \bydef{enhancement} of $\Lscr$.
They introduced a notion of an open book $\book{b}$
(or its fibered link $\Lscr_\book{b}$) \bydef{unfolding} 
into open books $\book{b_i}$ (or their fibered links
$\Lscr_\book{b_i}$), denoted by $\book{b} = 
\mathop\Upsilon_i \book{b_i}$ (or $\Lscr_\book{b} = 
\mathop\Upsilon_i \Lscr_\book{b_i}$);
with their definition, 
$\lambda\of{\Lscr_{\book{b_1}\mathop\Upsilon\book{b_2}}}
=\lambda\of{\Lscr_{\book{b_1}}}+\lambda\of{\Lscr_{\book{b_2}}}$
is tautologous.  They also show that unfolding includes 
Murasugi sum in the sense that for any open books 
$\book{b}_1, \book{b}_2$ on $S^3$, pages $F_i$ of 
$\book{b}_i$, and Murasugi sum $F=F_1\plumb{}F_2$, 
there exists an unfolding  $\book{b}=\book{b}_1\mathop\Upsilon\book{b}_2$
with $F$ as a page.  The generalization of
Proposition \ref{prop:lambda facts}%
\eqref{subprop:lambda additivity} from connected sum
to Murasugi sum follows immediately.

In \cite{NeumannRudolph1990} Neumann and 
Rudolph applied the calculus of splice diagrams 
\cite{EisenbudNeumann1985} to calculations of the enhancement
for various classes of fibered links.  In particular,
Proposition 9.3 of \cite{NeumannRudolph1990} (stated 
for a pair of coaxial torus \bydef{knots} but true 
for a pair of coaxial torus links in general) 
includes the following calculation, which can also 
be derived by a pleasant exercise using the techniques 
of \cite{Rudolph1987}.

\begin{proposition}\label{prop:lambda for Hpq}
For $p>q\ge0$, $\lambda(\Hopf{-}{p}{q})=2q-q^2$.\qed
\end{proposition}

At an Oberwolfach Research-in-Pairs-Workshop on 
$3$--manifolds and singularities convened (for 
a large value of ``pair'') by N.~A'Campo in 2000, 
several participants noticed simultaneously 
that when the map $L_{\Lscr_\book{p}}\from S^3\to S^2$ 
is taken to be a field of oriented tangent $2$--planes 
on $S^3$ (as in \cite{Rudolph1987}), it is clearly
isotopic to (and arbitrarily close to) a 
contact structure $\xi_{\book{b}}$ on $S^3$ 
for which $\xi_{\book{b}}$ and $\book{b}$ are 
\bydef{compatible} in the sense of Thurston and 
Winkelnkemper \cite{ThurstonWinkelnkemper1975} 
(alternatively, $\xi_{\book{b}}$ is \bydef{supported}
by $\book{b}$ in the sense of Giroux \cite{Giroux2002}) and  
the Hopf invariant of $\xi_\book{b}$ (as a plane field)
equals $\lambda(\Lscr_\book{b})$.

\begin{theorem}[Giroux \cite{Giroux2002}; see 
also Giroux and Goodman \cite{GirouxGoodman2006}]
\label{thm:Giroux-Goodman theorem}
Every contact structure on $S^3$ is ambient isotopic 
to a contact structure $\xi_{\book{b}}$ compatible with 
some open book $\book{b}$ on $S^3$; $\xi_{\book{b}_0}$ and 
$\xi_{\book{b}_1}$ are homotopic iff the fiber
surfaces of $\Lscr_{\book{b}_0}$ and $\Lscr_{\book{b}_1}$ 
are stably equivalent under the operation 
$F\mapsto F\plumb{P}\AKn{\Oscr}{-1}$
of {positive Hopf plumbing}. 
\qed
\end{theorem}
 
\begin{theorem}[\cite{Rudolph2005}; see also
Hedden \cite{Hedden2008}]%
\label{thm:qp fiber surfaces} 
A fibered link $\Lscr=\Lscr(\book{b})$ in $S^3$ is 
strongly quasipositive iff up to ambient isotopy 
$\book{b}$ is compatible with the standard, contact 
structure $\xi_0$ on $S^3$. \qed
\end{theorem}

In light of \ref{thm:qp fiber surfaces}, the following 
may be somewhat surprising.

\begin{proposition}\label{prop:overtwisted qp}
Every contact structure on $S^3$ is homotopic to
a contact structure $\xi_\book{b}$ compatible with 
an open book $\book{b}$ with non-strongly quasipositive
binding.
\end{proposition}

\begin{proof}
Let $q\ge 0$.  By \ref{cor:qp fibered Hpq} and 
\ref{prop:lambda for Hpq}, $\Hopf{-}{q+1}{q}$ is
a quasipositive fibered link with enhancement
$\lambda(\Hopf{-}{q+1}{q})=2q-q^2$.  In particular,
$\lambda(\Hopf{-}{2}{1})=1$ and 
$\lambda(\Hopf{-}{q+1}{q})\searrow-\infty$ as
$q\nearrow\infty$.  It follows from 
\ref{prop:lambda facts}\eqref{subprop:lambda additivity}
that $\lambda$ achieves every integer value on an
appropriate connected sum
\begin{equation*}
\Lscr_{q,m}=\Hopf{-}{q+1}{q}\connsum\overbrace{\Hopf{-}{2}{1}\connsum
	\cdots\connsum\Hopf{-}{2}{1}}^\text{$m$ times}.
\end{equation*}
Connected sum preserves both quasipositivity and fiberedness, 
so by Theorem \ref{thm:Giroux-Goodman theorem} and the 
paragraph that precedes it the proof is complete except for 
the homotopy class of $\xi_0$.  That case is covered by  
observing that $\Hopf{-}{3}{2}$, though quasipositive,
is not strongly quasipositive---for 
instance (as illustrated in Figure~\ref{fig:qpfibered}) 
because it is realized as a $1$--dimensional 
transverse $\C$--link by the link at infinity of  
$z_0(z_0z_1-1)$: since the Euler characteristic $-1$ of 
its fiber surface (a pair of pants) is strictly smaller
than that of its $\C$--span (the disjoint union of an
annulus and a disk), the truth of the Thom Conjecture
(Kronheimer and Mrowka \cite{KronheimerMrowka1994})
implies that the fiber surface is not quasipositive. 
\end{proof}

\begin{figure}[ht!]
\begin{center}
\includegraphics[width=\figwidth]{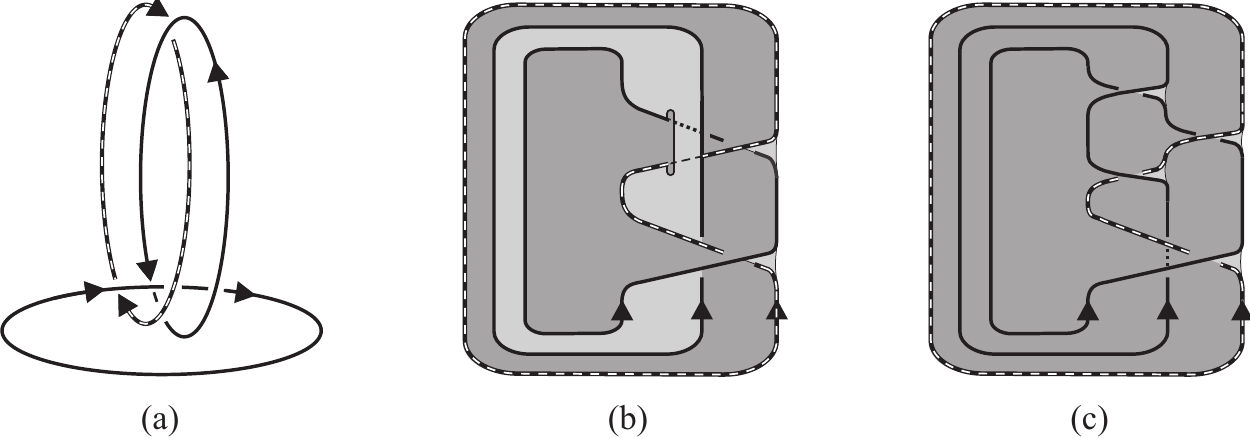} 
\caption{(a)~$\Hopf{-}{3}{2}=\Clinkinfty{z_0(z_0z_1-1)}$.
(b)~The $\C$--span of $\Clinkinfty{z_0(z_0z_1-1)}$ 
in $\DSK{4}{\zero}{1/\epsilon}$, represented 
as a ribbon-immersed surface in $S^3$. 
(Figure \ref{fig:Hopf links}(d) suggests the analogous
surface for any $\Hopf{-}{p}{p-1}$ with $p\ge 3$.)  
(c)~The fiber surface of $\Clinkinfty{z_0(z_0z_1-1)}$.
\label{fig:qpfibered}}
\end{center}
\end{figure}

\subsection{Construction: quasipositive links
with distinct $\C$--spans (new)}
Consider the quasipositive band
representations 
\begin{align*}
\brep\rho_0\isdefinedas &
(\SideSet{\sigma_3^{-2}\sigma_2}{\sigma_1},
\sigma_2,
\SideSet{\sigma_1 \sigma_3^3}{\sigma_2},
\SideSet{\sigma_1 \sigma_3^5 \sigma_2^{-1}}{\sigma_1},
\sigma_3,
\sigma_3),
\\
\brep\rho_1\isdefinedas &
(\sigma_2,
\SideSet{\sigma_1\sigma_3}{\sigma_2},
\SideSet{\sigma_1\sigma_3}{\sigma_2},
\SideSet{(\sigma_1\sigma_3)^2}{\sigma_2},
\SideSet{(\sigma_1\sigma_3)^2}{\sigma_2},
\SideSet{(\sigma_1\sigma_3)^3}{\sigma_2})
\end{align*}
in $B_4$ and their associated quasipositive braided 
surfaces realized (as in \ref{subsect:braids})
by $\C$--spans of $1$--dimensional transverse $\C$--links
$\Clink{g_i}{S^3}$ in $S^3$.  Auroux, Kulikov, and Shevchishin 
\cite{Aurouxetal2004} show that, although the braids
$\braidof{\brep\rho_0}$ and $\braidof{\brep\rho_1}$
are equal, and $\Cspan{g_0}{D^4}$ is diffeomorphic to
$\Cspan{g_1}{D^4}$ (both are twice-punctured tori),
$D^4\setminus \Cspan{g_0}{D^4}$ is not homeomorphic 
to $D^4\setminus \Cspan{g_1}{D^4}$ (their fundamental
groups are different).  In particular, although 
$\Clink{g_0}{S^3}$ and $\Clink{g_1}{S^3}$ are
ambient isotopic as smooth links in $S^3$, 
their $\C$--spans are not ambient isotopic as
smooth $2$--submanifolds-with-boundary in $D^4$.
By appending $\SideSet{\sigma_1^3}{\sigma_2}$
to $\rho_0$ and $\rho_1$, Geng \cite{Geng2011} showed 
that even smoothly isotopic quasipositive \textit{knots} 
can have (non-singular, diffeomorphic) 
$\C$--spans that are not ambient isotopic in $D^4$
(again, the fundamental groups of their complements
are not isomorphic).

\begin{remark}\label{rmk:cutting out and pushing off annuli}
Another, easier construction produces arbitrarily large 
finite sets of mutually ambient isotopic $1$--dimensional 
transverse $\C$--links in $S^3$ (in fact, strongly 
quasipositive links) with pairwise non-diffeomorphic 
$\C$--spans (all of the same Euler characteristic); 
the simplest example, shown in Figure \ref{fig:subsurf}, 
suffices to illustrate the general method, which 
necessarily produces link-manifolds of at least $3$
components.\done
\end{remark}

\begin{figure}[ht!]
\begin{center}
\includegraphics[width=\figwidth]{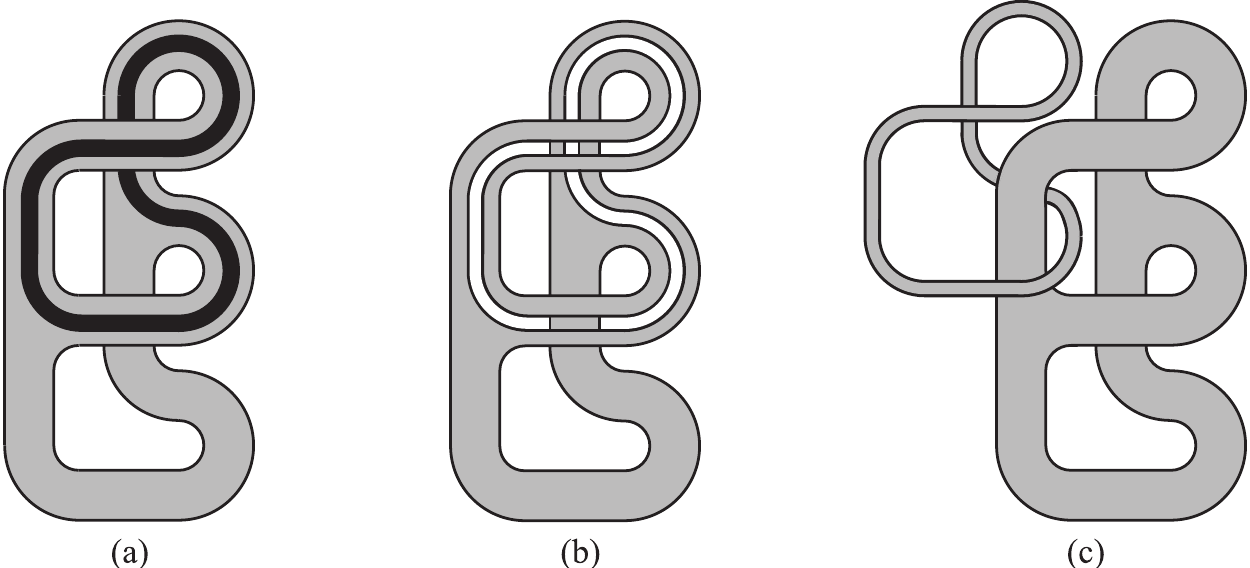} 
\caption{(a)~A connected quasipositive Seifert surface $S$ 
(gray) and an annular subsurface $F\sub S$ (black). 
(b)~The cut-open Seifert surface $S\setminus\Int{F}$ 
is again quasipositive, by \ref{prop:full on qp Ss}.  
(c)~The Seifert surface $S\cup\pushoff{F}$,
comprising $S$ and the push-off of $F$, 
is quasipositive (again, by \ref{prop:full on qp Ss}) 
and $(\Bd(S\cup\pushoff{F}),S^3)$ is ambient isotopic
to $(\Bd(S\setminus\Int F),S^3)$ (by an isotopy that
begins by rotating each component of $\pushoff{F}$ 
around a core $S^1$ so as to interchange its two
boundary components).
\label{fig:subsurf}}
\end{center}
\end{figure}

\subsection{Construction: quasipositive orientations of unoriented links (new) \label{subsect:qp orientations of unoriented links}}
Given an oriented manifold $L$ with $n\ge 1$ components, 
the unoriented manifold $\abs{L}$ supports $2^n$ orientations 
and thus $2^{n-1}$ \bydef{projective orientations}, each determined 
by an orientation and its componentwise opposite (Sakuma
\cite{Sakuma1994} uses the term \bydef{semi-orientation} for
this concept).  Write $\orientation$ for a projective orientation.  

Let $\Lscr=(L,S^3)$ be an oriented classical link (with,
as is usual, the orientation of $L$ not included explicitly 
in the notation).  As noted in \ref{subsubsect:positive links},
$\Lscr$ is positive iff its \bydef{opposite}
$-\Lscr\isdefinedas(-L,S^3)$ is positive,
the proof being consideration of any link diagram of $\Lscr$.
Similarly, $\Lscr$ is quasipositive iff $-\Lscr$ is;
here the proof is to note that reversing the 
orientation of a braid diagram with closure $\Lscr$,
then rotating it by $\pi$ in its plane, makes it into
a braid diagram with closure $-\Lscr$, and that this
operation preserves diagrammatic quasipositivity.
One might expect that at most one projective orientation 
of an unoriented classical link makes it quasipositive,
and that is the case with a few exceptions 
(eg, a trivial knot;
split links of two or more positive knots; 
$\abs{\Hopf{\pm}{2}{0}}=\abs{\Hopf{\pm}{1}{1}}$).

This section collects 
several useful examples of families of unoriented 
classical links in which each member supports a 
projective orientation (typically but not invariably 
unique) that makes it quasipositive---briefly, a 
\bydef{quasipositive orientation}.

\subsubsection{Quasipositive orientations of unknotted strip
boundaries\label{subsubsect:qp strip boundaries}}

Let $\Kscr$ be a classical knot, $t$ an integer,
and $\strip{\Kscr}{t}$ the strip of type $\Kscr$ with
$t$ half-twists as defined in \ref{def:strips}.
The unoriented link $(\Bd\strip{\Kscr}{t},S^3)$ 
has $1$ or $2$ components according as $t$ is odd or even.
In both cases, let $\orientation$ be the ``braidlike''
projective orientation, so that 
$(\Bd\strip{\Kscr}{t}^\orientation,S^3)\defines\Kscr\{2,t\}$ 
is the 
\bydef{$2$--strand cable on $\Kscr$ with $t$ half-twists}.  
(Both the notation $\Kscr\{m,n\}$, already introduced for $n=0$, 
and its iterated extension 
$\Kscr\{m_1,n_1;m_2,n_2;\dots;m_q,n_q\}\isdefinedas
\Kscr\{m_1,n_1\}\{m_2,n_2\}\cdots\{m_q,n_q\}$,
are adapted from Litherland \cite{Litherland1979};
this is reasonably consistent with the notation
for satellites.) In case $t$ is even, let 
$\orientation^\prime$ be the ``non-braidlike'' projective 
orientation.  (See Figure \ref{fig:strip-boundaries}(a) and (c).)

\begin{figure}[ht!]\begin{center}
\includegraphics[width=\figwidth]{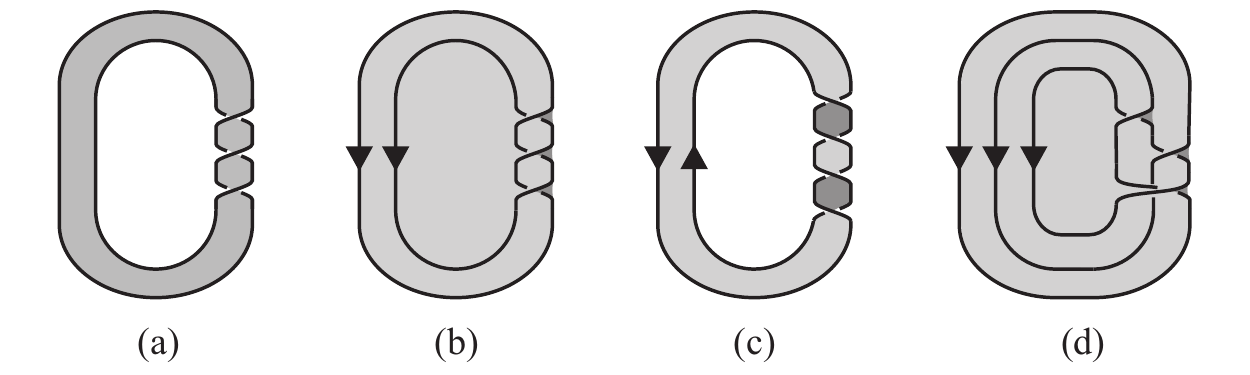} 
\caption{(a)~The unoriented unknotted strip 
$\strip{\Oscr}{3}$ with $3$ half-twists, and its unoriented 
boundary.
(b)~For $t>0$, $\Bd\strip{\Oscr}{t}^\orientation$ bounds
the quasipositive braided Seifert surface 
$S(\sigma_1,\dots,\sigma_t)$ with all $b_i=\sigma_1\in B_2$.
(c)~For $s<0$, $\Bd\strip{\Oscr}{2s}^{\orientation^\prime}$
bounds the quasipositive annular Seifert surface 
$\AKn{\Oscr}{s}$.
(d)~$\AKn{\Oscr}{-2}$ as the quasipositive braided Seifert 
surface $S(\SideSet{\sigma_1}\sigma_{2},\sigma_{2},\sigma_1)$.
\label{fig:strip-boundaries}}
\end{center}
\end{figure}

\begin{proposition}\label{prop:qp unknotted strip boundaries}
\begin{inparaenum}[\upshape(1)]
\item\label{subprop:qp unknotted braidlike strip boundaries}
$(\Bd\strip{\Oscr}{t}^\orientation,S^3)=\tor{2}{t}$ is 
(strongly) quasipositive iff $t\ge 0$.
\item\label{subprop:qp unknotted nonbraidlike strip boundaries}
$(\Bd\strip{\Oscr}{t}^{\orientation^\prime}%
\kern-4pt,S^3)$ is 
(strongly) quasipositive iff $t=2s\le 0$;
then it is $(\Bd\AKn{\Oscr}{s},S^3)$.\qed
\end{inparaenum}
\end{proposition}

\begin{remark}\label{rmk:qp knotted strip boundaries}
Similar results for $\Kscr\ne\Oscr$ are true but more 
complicated to state.\done
\end{remark}

\subsubsection{Quasipositive rational links 
\label{subsubsect:qp rational links}}
The torus link $\tor{2}{k}$ in 
\ref{subsubsect:qp strip boundaries} is well
known to be a fibered link for $k\ne 0$; its 
fiber surface is the braided Seifert surface 
$S=S(\sigma_1^{\sgn(k)},\dots,\sigma_t^{\sgn(k)})$
(with $\abs{k}$ bands), illustrated for $k=3$ in 
Figure \ref{fig:strip-boundaries}(b).  In fact, 
$\tor{2}{k}$ is both a Hopf-plumbed link
as defined in 
\ref{defs:iterated strip plumbing}\eqref{def:Hopf plumbing}
and---with its orientation forgotten---an arborescent link 
as defined in \ref{defs:iterated strip plumbing}
\eqref{def:arborescent plumbing, traditional version}.  
More precisely, in this last guise $\abs{\tor{2}{k}}$ is
an unoriented \bydef{rational link} 
\begin{equation*} 
\rational{\overbrace{-2,-2,\dots,-2}
^\text{$k-1$}}\isdefinedas
(\rationallink{\overbrace{-2,-2,\dots,-2}%
^\text{$k-1$}},S^3)
\end{equation*}
where, for $r_1,\dots,r_n\in \Z\setminus\{0\}$,
$\rationallink{r_1,r_2,\dots,r_n}$ denotes the boundary 
of a $2$--manifold-with-boundary 
$\rationalsurface{r_1,r_2,\dots,r_n}$ 
strip-plumbed as in Figure \ref{fig:rational links}(a)
according to a \bydef{stick}---that is, a 
\bydef{tree} (a finite connected 
acyclic $1$--dimensional simplicial complex)
with no \bydef{nodes} (vertices of valence $3$ or 
greater) equipped with a weighting of its vertices
by integers;  
Figure \ref{fig:rational links}(b) is a standard 
depiction of a stick.
\begin{figure}[ht!]
\begin{center}
\includegraphics[width=\figwidth]{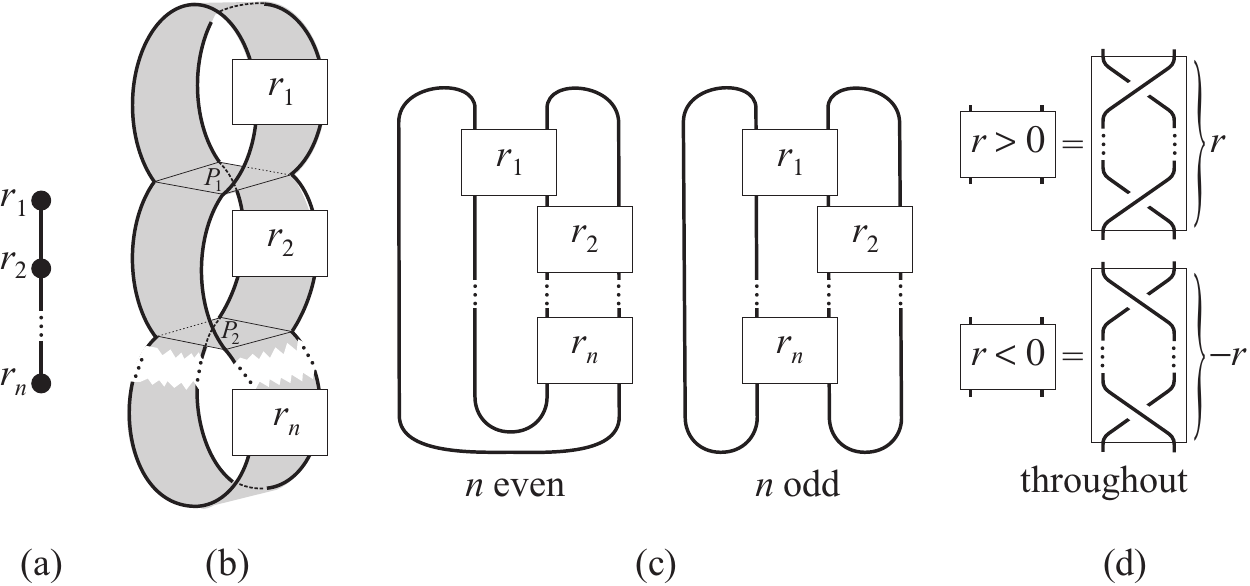} 
\caption{(a)~A stick $\stickgraph{r_1,r_2,\dots,r_n}$.
(b)~The strip-plumbed surface 
$\rationalsurface{r_1,r_2,\dots,r_n}$ $\sub S^3$.
(c)~The rational link-manifold $\rationallink{r_1,r_2,\dots,r_n}$
presented as the $4$--plat of  
$\sigma_2^{r_1}\sigma_3^{r_2}\cdots\sigma_3^{r_n}$ ($n$ even)
or $\sigma_2^{r_1}\sigma_3^{r_2}\cdots\sigma_2^{r_n}$ ($n$ odd),
read top to bottom, with plat closure as indicated.
(d)~Sign conventions for the $2$--string tangles in
(b), (c), and elsewhere.
\label{fig:rational links}}
\end{center}
\end{figure}

Clearly $\rationalsurface{r_1,r_2,\dots,r_n}$ 
is orientable iff all $r_i$ are even, and then 
$\rational{r_1,\dots,r_n}$ has a preferred 
projective orientation $\orientation$.
If also $r_i<0$ for all $i$, 
then (by \ref{thm:qp annuli} and \ref{thm:qp Murasugi sum theorem}) 
$\rationalsurface{r_1,\dots,r_n}$
is a quasipositive Seifert surface, and
$\rational{r_1,\dots,r_n}^\orientation$ 
is strongly quasipositive. 
But these sufficient conditions for 
$\rational{r_1,\dots,r_n}$ to have
a (strongly) quasipositive orientation are far 
from necessary.  
The following is true by inspection.

\begin{proposition}\label{prop:2-bridge machine}
The rational link-manifold $\rationallink{r_1,\dots,r_n}$ 
has a projective orientation $\orientation$ which, applied to
the $4$--plat diagram in Figure 
\textup{\eqref{fig:rational links}(c)},
makes it a positive diagram iff the 
braid $\sigma_2^{r_1}\sigma_3^{r_2}\sigma_3^{r_3}\cdots%
\sigma_\ell^{r_n}\in B_4$ (with $\ell$ equal to $3$ or $2$ 
according as $n$ is even or odd) is generated by the 
labeled digraph in Figure \textup{\ref{fig:2-bridge machine}}.
\qed \end{proposition}

Here, $\beta\in B_4$ is \bydef{generated} by the labeled 
digraph in case there is a directed path from one of the 
(source) boxes at the top of the digraph to one of the 
(sink) boxes at the bottom of the digraph such that $\beta$ 
is produced by first concatenating the labels on the labeled 
edges of the path and then replacing each instance of the
letter ``$\mathrm{a}$'' (respectively 
``$\mathrm{e}$'' or ``$\mathrm{o}$'') by an
arbitrary (respectively even or odd) strictly 
positive integer.

\begin{figure}[ht!]
\begin{center}
\includegraphics[width=\figwidth]{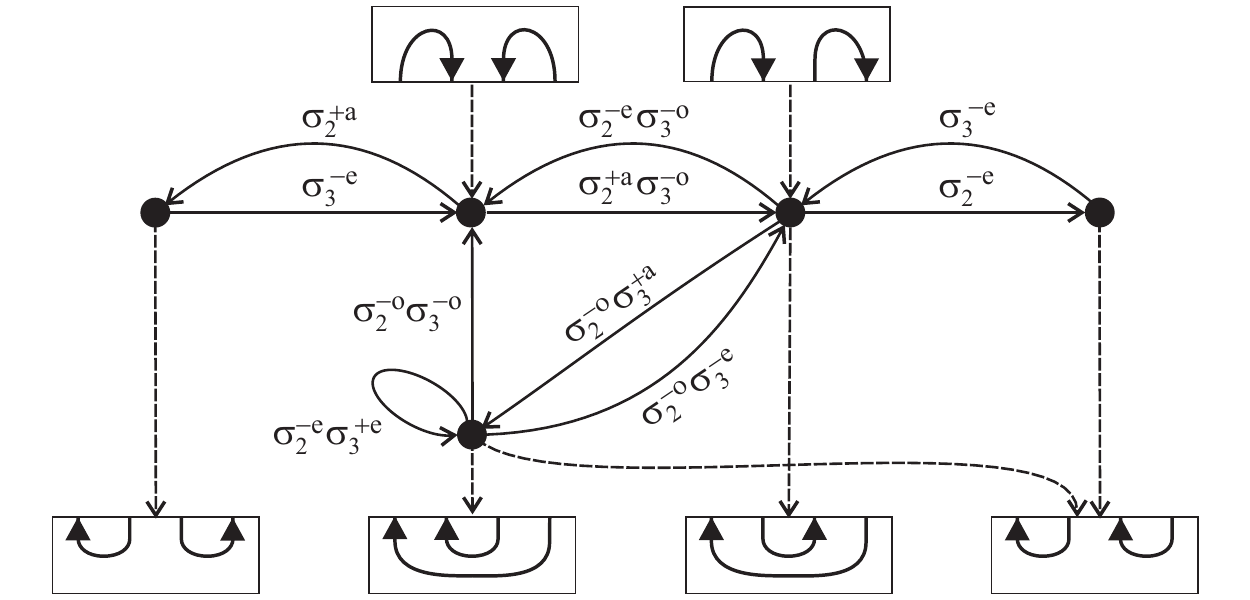} 
\caption{This machine generates all positive oriented 
rational links as $4$--plats.
\label{fig:2-bridge machine}}
\end{center}
\end{figure}

\begin{questions}\label{qn:rational machine questions} 
Proposition \ref{prop:2-bridge machine} gives expansive, but 
imperspicuous, sufficient conditions 
for $\rational{r_1,r_2,\dots,r_n}$ to have
a (strongly) quasipositive orientation. 
\begin{inparaenum}[(1)]
\item\label{subqn:language of 2-bridge machine}
What is a closed-form description (presumably one exists) 
of the set of rational numbers 
\begin{equation*}
r_1+
\cfrac{1}{-r_2+
\cfrac{1}{\dotsb+
\cfrac{1}{(-1)^{n-1}r_n
}}}
\end{equation*}
such that $\sigma_1^{r_1}\sigma_2^{r_2}\cdots%
\sigma_{(3\pm1)/2}^{r_n}$ is generated as in 
\ref{prop:2-bridge machine}? (See \ref{subsubsect:lens spaces}.)
\item\label{subqn:NASC for (s)qp rational}
Are the necessary and sufficient conditions for positivity
given in \ref{prop:2-bridge machine}
also necessary for strong quasipositivity?  For 
quasipositivity?  My tentative answers are ``probably
yes'' and ``almost certainly no''.\;\done
\end{inparaenum}
\end{questions}

\begin{remark}\label{rmk:stick origin}
The term ``stick'' is due to Bonahon and Siebenmann
\cite{BonahonSiebenmann1979/2010} (but there a stick 
may lack one or both terminal vertices yet retain 
its terminal edge or edges).
\done
\end{remark}

\subsubsection{Quasipositive pretzel links}%
\label{subsubsect:quasipositive pretzel links}
Another guise in which $\abs{\tor{2}{k}}$ ($k>0$) appears 
is as the \bydef{unoriented pretzel link}
\begin{equation*} 
\pretzel{\overbrace{-1,\dots,-1}^\text{$k$ times}}
\isdefinedas
(\pretzellink{\overbrace{-1,\dots,-1}^\text{$k$ times}},
S^3)
\end{equation*}
where, for $t_1,t_2,\dots,t_p\in\Z$, 
$\pretzellink{t_1,t_2,\dots,t_p}$ is the unoriented 
boundary of two unoriented $2$--submanifolds-with-boundary 
of $S^3$, depicted in
Figure \ref{fig:general pretzel surface and pretzel plumbing}%
(b) and (c).  

\begin{definitions}\label{defs:spans of pretzel links}
\begin{inparaenum}[(1)]
\item\label{def:star surfaces}
The \bydef{star surface} 
$\starplumb{0}{t_1,\dots,t_{p}}$ is strip-plumbed 
according to the \bydef{star} $\stargraph{0}{t_1,\dots,t_{p}}$,
where in general $\stargraph{c}{t_1,\dots,t_{p}}$
has a central node of weight $c$ and $p\ge 3$ 
\bydef{twigs} (terminal vertices) weighted $t_1,\dots,t_p$ 
in the cyclic order determined by some planar embedding 
of a geometric realization, as depicted in 
Figure \ref{fig:general pretzel surface and pretzel plumbing}(b)).
\item\label{def:pretzel surface}
The \bydef{pretzel surface} $\pretzelsurface{t_1,\dots,t_p}$
is defined by its ordered handle decomposition into 
two $0$--handles lying on $S^2\sub S^3$ and 
$p\ge 3$ $1$--handles with core arcs lying
on $S^2$, each of them joining the two $0$--handles, and
such that the $i$th $1$--handle has \bydef{twisting number}
$t_k\in\Z$ (normalized so that, eg, 
$\pretzelsurface{-1,-1,-1}=\abs{\tor{2}{3}}$).
\done
\end{inparaenum}
\end{definitions}
\begin{figure}[ht!]
\centering
\includegraphics[width=\figwidth]%
{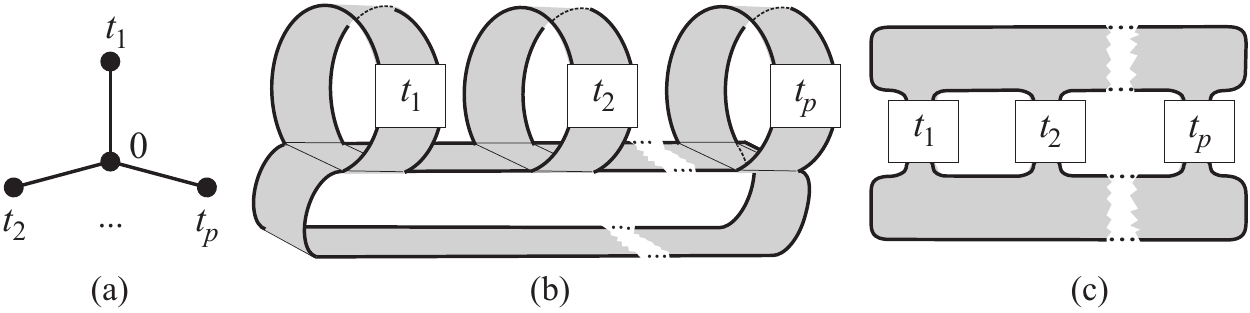} 
\caption{(a)~The star $\stargraph{0}{t_1,t_2,\dots,t_p}$.
(b)~The star surface 
$\stripplumb{\stargraph{0}{t_1,t_2,\dots,t_p}}$.
(c)~The pretzel surface $\pretzelsurface{t_1,\dots,t_{p}}$.
\label{fig:general pretzel surface and pretzel plumbing}}
\end{figure}
\begin{remarks}\label{rmks:star surfaces}
\begin{inparaenum}[(1)]
\item\label{rmk:non-qp star surfaces}
By \ref{prop:full on qp Ss}, \ref{thm:qp annuli},
and the non-orientability of $\strip{\Oscr}{t}$ for
odd $t$, a star surface is orientable and has a quasipositive
orientation iff all weights are even and 
strictly negative.  In particular, 
the central plumband $\strip{\Oscr}{0}$ of 
$\starplumb{0}{t_1,\dots,t_p}$ keeps it from having
a quasipositive orientation.
\item\label{rmk:trading off ones}
For $c\ne 0$, $\starplumb{c}{t_1,\dots,t_{p}}$ and 
$\starplumb{c}{t_1,\dots,t_p,t_{p+1},\dots,t_{p+\abs{c}}}$,
where $t_{p+j}=-\sgn{c}$ for $j=1,\dots,\abs{c}$,
have ambient isotopic boundaries.
\end{inparaenum}
\done
\end{remarks}

The first claim in the following proposition is obvious;
the necessity of the second claim follows from 
\ref{prop:full on qp Ss}, and its sufficiency was 
proved by Rudolph \cite{Rudolph2001}.

\begin{proposition}\label{prop:qp pretzel surface}
\begin{inparaenum}[\upshape(1)]
\item\label{subprop:pretzel surface parity condition}
$\pretzelsurface{t_1,\dots,t_{p}}$ is orientable 
(with unique projective orientation $\orientation$) 
iff all $t_i$ have the same parity.
\item\label{subprop:pretzel surface negative sum condition} 
If $\pretzelsurface{t_1,\dots,t_{p}}$ is orientable,
then $\pretzelsurface{t_1,\dots,t_{p}}^\orientation$
is quasipositive iff $t_i+t_j<0$ for $1\le i<j\le p$.  
\end{inparaenum}\qed
\end{proposition}

Proposition \ref{prop:qp pretzel surface} 
is not the whole story on pretzel links with (strongly)
quasipositive orientations.  What was overlooked in
\cite{Rudolph2001} was that there are many $(t_1,\dots,t_p)$
failing the parity condition 
\ref{prop:qp pretzel surface}%
\eqref{subprop:pretzel surface parity condition},
the negative-sum condition
\ref{prop:qp pretzel surface}%
\eqref{subprop:pretzel surface negative sum condition},
or both, for which there nonetheless exists $\orientation$ 
making $\pretzel{t_1,t_2,\dots,t_p}^\orientation$ 
quasipositive.  This can happen in (at least) two ways.
\begin{inparaenum}[(a)]
\item\label{subex:ribbon pretzel knots}
$\pretzellink{t_1,\dots,t_p}^\orientation$ may bound
a quasipositive non-embedded ribbon-immersed surface
in $S^3$, but not bound any quasipositive Seifert surface,
making $\pretzel{t_1,\dots,t_p}^\orientation$ 
quasipositive but not strongly quasipositive;
the simplest knot of this type is 
$\pretzel{-3,3,-2}^\orientation$,
depicted in Figure \ref{fig:ribbon and positive pretzels}(a).
\item\label{subex:positive pretzel knots}
$\pretzellink{t_1,\dots,t_p}^\orientation$ may be a 
positive link, and thus strongly quasipositive 
by \ref{thm:positive implies stqp}; this case 
is addressed by the following proposition. 
\end{inparaenum}
\begin{figure}[ht!]
\centering
\includegraphics[width=\figwidth]{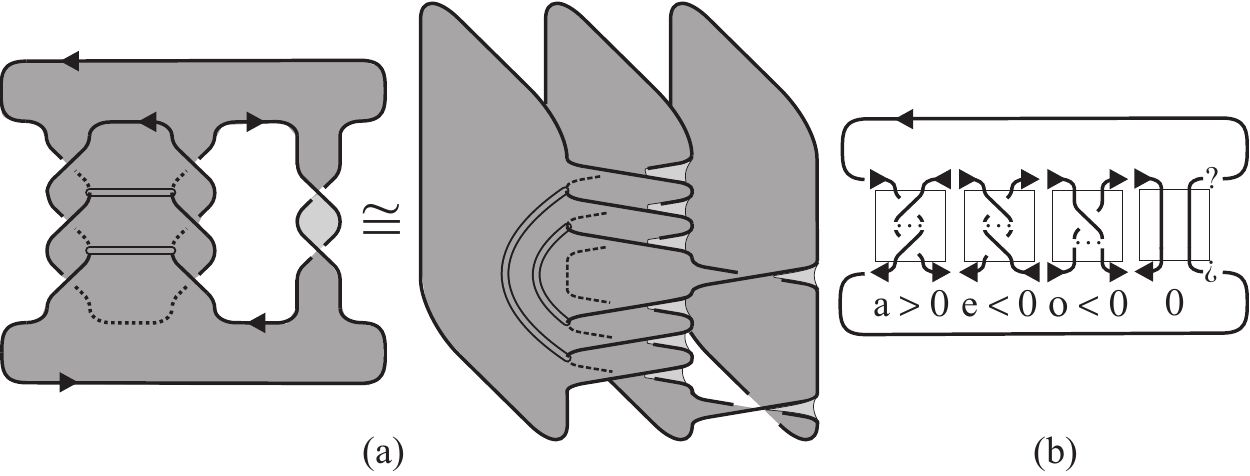} 
\caption{(a)~Isotopic ribbon-immersed surfaces bounded by 
$\pretzellink{3,-3,-2}^\orientation$ and the closure
of the quasipositive braid 
$\sigma_2\sigma_1^3\sigma_2\sigma_1^{-3}\in B_3$.
(b)~Templates for positive pretzels.
\label{fig:ribbon and positive pretzels}}
\end{figure}

\begin{proposition}\label{prop:positive pretzel links}
The following are equivalent.
\begin{inparaenum}[\upshape(A)]
\item\label{subprop:positive diagram}
$\pretzellink{t_1,\dots,t_p}$ has a projective orientation $\orientation$ such that the oriented link diagram of 
$\pretzel{t_1,\dots,t_p}^\orientation$ implicit in Figure~%
\textup{\ref{fig:general pretzel surface and pretzel plumbing}(c)} 
is positive.
\item\label{subprop:pretzel tangles}
The positive projective orientations of the 
$p$ $2$--string tangles indicated by the boxes 
labeled $t_1,\dots,t_p$ in Figure~%
\textup{\ref{fig:general pretzel surface and pretzel plumbing}(c)} 
are consistent.
\item\label{subprop:signs and parities}
Either 
	\begin{inparaenum}[\upshape(a)]
	\item\label{subsubprop:all-odd positive pretzel links}
	all $t_i$ are odd and negative, or
	\item\label{subsubprop:odd and even positive pretzel links}
	no odd $t_i$ is negative, and an even number
	of $t_i$ are strictly positive.
	\end{inparaenum}
\end{inparaenum}
\end{proposition}

\begin{proof}\label{pf:positive pretzel links}
It is clear that \eqref{subprop:positive diagram} 
and \eqref{subprop:pretzel tangles} are equivalent.
The equivalence of \eqref{subprop:pretzel tangles}
and \eqref{subprop:signs and parities} follows by
considering how the schematic templates for positive
oriented $2$--tangles shown in
Figure \ref{fig:ribbon and positive pretzels}(b) 
can fit together maintaining projectively consistent  
orientations~with each other and with the trivial 
$2$--tangle comprising the top and bottom of the
diagram.
\end{proof}

\begin{questions}\label{qn:qp pretzel questions} 
\begin{inparaenum}[(1)]
\item\label{subqn:non-strongly qp pretzel links}
The example depicted in 
Figure \ref{fig:ribbon and positive pretzels}(a)
can be generalized somewhat (eg, to $P(2n+1,-(2n+1),-2m)$ 
for all $m,n>0$), but it is not immediately
clear just how far.  Are there useful criteria for 
a pretzel link to have a quasipositive orientation that 
is not strongly quasipositive?  What about 
pretzel knots?  
\item\label{subqn:other stqp pretzel links?}
Excepting the case (covered by 
Proposition \ref{prop:qp pretzel surface})
in which $t_i+t_j$ is even and strictly negative for 
$1\le i<j\le p$, and one $t_i$ is non-negative, 
can $\pretzel{t_1,\dots,t_p}$ have a strongly quasipositive
orientation $\orientation$ that does not make 
$\pretzel{t_1,\dots,t_p}^\orientation$ actually a 
positive link as in Proposition 
\ref{prop:positive pretzel links}?
\end{inparaenum}\done
\end{questions}

\subsubsection{Strongly quasipositive arborescent links 
\label{subsubsect:sqp arborescent links}}

Three general methods produce large families 
of arborescent links supporting strongly quasipositive 
orientations; I do not know whether
all such links are produced in one of these ways.  To 
that extent (if not further: cf \ref{qn:qp pretzel questions}%
\eqref{subqn:non-strongly qp pretzel links}) this 
subsection is work in progress.

\begin{definitions}
Let $(\Ts,w)$ be a weighted tree.
\begin{inparaenum}[(1)]
\item\label{def:stqp (T,w)}
Call $(\Ts,w)$ \bydef{strongly quasipositive} 
in case there exists a projective orientation 
$\orientation$ of $\Bd\stripplumb{\Ts,w}$ such that
$(\Bd\stripplumb{\Ts,w}^\orientation,S^3)$
is a strongly quasipositive link.
\item\label{def:positive (T,w)}
Call $(\Ts,w)$ \bydef{positive} 
in case there exists a projective orientation 
$\orientation$ of $\Bd\stripplumb{\Ts,w}$ such that
the canonical unoriented link diagram of
$(\Bd\stripplumb{\Ts,w}^\orientation,S^3)$ 
(Gabai's \bydef{$\mathsf{T}$--projection}
\cite[Figure 1.4]{Gabai1986}; see also 
Bonahon and Siebenmann 
\cite[Figure 12.12]{BonahonSiebenmann1979/2010}), when endowed
with $\orientation$, becomes a positive oriented link diagram.
\item\label{def:vstqp (T,w)}
Call $(\Ts,w)$ \bydef{very strongly quasipositive} 
in case there exists a projective orientation 
$\orientation$ of $\stripplumb{\Ts,w}$ such that
$\stripplumb{\Ts,w}^\orientation$
is a quasipositive Seifert surface.
\end{inparaenum}\done
\end{definitions}

\eqref{def:positive (T,w)} and \eqref{def:vstqp (T,w)} 
each imply \eqref{def:stqp (T,w)}; 
\eqref{def:stqp (T,w)}, \eqref{def:positive (T,w)},
and \eqref{def:vstqp (T,w)} have no other non-trivial
implications.

To explore these properties, a few more definitions are useful. 
Let $\Ts$ be a planar tree with vertex set $\Verts{\Ts}$,
$w\from\Ts\to\Z$ a weighting. 
Writing $d(v,v')$ for the number of edges in the simple 
edge-path in $\Ts$ joining $v,v'\in\Verts{\Ts}$, call 
$v$ and $v'$ \bydef{adjacent} in case $d(v,v')=1$ 
and \bydef{distant} in case $d(v,w)\ge 3$.  A vertex 
adjacent to at least three is a node of $\Ts$,
and a vertex adjacent to at most one vertex is a twig 
of $\Ts$, as previously defined. 

\begin{definitions}\label{defs:decomposing a weighted tree}
Let $u$ be a non-node and $v$ a node of $\Ts$.
\begin{inparaenum}[(1)]
\item\label{def:stick of a non-node}
Denote by $\stickof{u}{\Ts}$ the subtree of $\Ts$ such 
that $u'\in\Verts{\stickof{u}{\Ts}}$ iff no vertex of the 
simple edge-path in $\Ts$ joining $u$ to $u'$ is a node.
The weighted tree $(\stickof{u}{\Ts},w\restr\stickof{u}{\Ts})$ is 
\bydef{the stick of $u$ in $(\Ts,w)$}; 
it is isomorphic to $\stickgraph{w(u'_1),\dots,w(u'_n)}$, where
$u'_1$ and $u'_n$ are the twigs of $\stickof{u}{\Ts}$
(so $v'_1=v'_n$ iff $\stickof{u}{\Ts}$ has $0$ edges)
and $d(v'_1,v'_q)=q-1$ for $q=1,\dots,n$.
\item\label{def:star of a node}
Denote by $\starof{v}{\Ts}$ the subtree of $\Ts$ such 
that $v'\in\Verts{\starof{v}{\Ts}}$ iff $d(v,v')\le 1$.
The weighted tree $(\starof{v}{\Ts},w\restr\starof{v}{\Ts})$ 
is \bydef{the star of $v$ in $(\Ts,w)$}; 
it is isomorphic to $\stargraph{w(v)}{w(v'_1),\dots,w(v'_p)}$, 
where $v'_1,\dots,v'_p$ are the twigs of $\starof{v}{\Ts}$
enumerated consistently with the planar embedding of $\Ts$.
\end{inparaenum}\done
\end{definitions}

\begin{lemma}\label{lemma:positive sticks and stars}
\begin{inparaenum}[\upshape(1)]
\item\label{sublemma:merely positive sticks}
$\stickgraph{r_1,\dots,r_n}$ is positive iff the braid 
$\sigma_2^{r_1}\sigma_3^{r_2}\sigma_3^{r_3}%
\cdots\sigma_\ell^{r_n}$ is generated by the machine 
in Figure~\upref{fig:2-bridge machine}.
\item\label{sublemma:vstqp sticks}
$\stickgraph{r_1,\dots,r_n}$ is very strongly quasipositive iff 
$r_i$ is even and strictly negative for $i=1,\dots,n$ iff
$\sigma_2^{r_1}\sigma_3^{r_2}\sigma_3^{r_3}\cdots\sigma_\ell^{r_n}$
is generated by the submachine obtained by deleting all but the
five rightmost arrows in Figure~\upref{fig:2-bridge machine}.
\item\label{sublemma:positive pretzel surfaces}
If $t_i+t_j$ is even and strictly negative
for $1<i\le j<p$, then $\stargraph{0}{t_1,\dots,t_p}$ 
is strongly quasipositive; if also $t_i<0$ for all $i$,
then $\stargraph{0}{t_1,\dots,t_p}$ 
is positive.
\item\label{sublemma:positive pretzel links}
If no odd $t_i$ is negative, and an even number
of $t_i$ are strictly positive, then 
$\stargraph{0}{t_1,\dots,t_p}$ is positive.
\item\label{sublemma:positive positive non-pretzel stars}
If $c>0$, all $t_i$ are odd and negative, and $c+p$ is even,
then $\stargraph{c}{t_1,\dots,t_p}$ is positive.
\item\label{sublemma:positive negative non-pretzel stars}
If $c$ and all $t_i$ are even and strictly negative, 
then $\stargraph{c}{t_1,\dots,t_p}$ is very strongly quasipositive.	\end{inparaenum}\qed
\end{lemma}

\begin{theorem}\label{thm:positive weighted trees}
If $(\Ts,w)$ is such that 
\begin{inparaenum}[\upshape(1)]
\item\label{subthm:no 0 non-nodes}
if $u$ is a non-node then $w(u)\ne 0$, and 
\item\label{subthm:0 nodes are distant}
if $v, v'$ are adjacent nodes, then $w(v)$ or
$w(v')$ is non-zero,
\end{inparaenum} 
then $(\Ts,w)$ is positive iff the stick of every non-node in 
$(\Ts,w)$ and the star of every node in $(\Ts,w)$ is positive.
\end{theorem}

\begin{proof}
Unless $n=0$, exactly one projective orientation
of the canonical $2$--string tangle depicted in Figure 
\ref{fig:rational links}(d) makes all $\abs{n}$ crossings
positive. \eqref{subthm:no 0 non-nodes} and
\eqref{subthm:0 nodes are distant} ensure that the
positive sticks and stars fit together consistently at
the appropriate twigs of each. 
\end{proof}

\begin{remark}
\eqref{subthm:no 0 non-nodes} can be weakened but
not dispensed with entirely; see Figure~\ref{fig:adjacent 0-nodes}.\done
\end{remark}

\begin{figure}[ht!]
\begin{center}
\includegraphics[width=\figwidth]{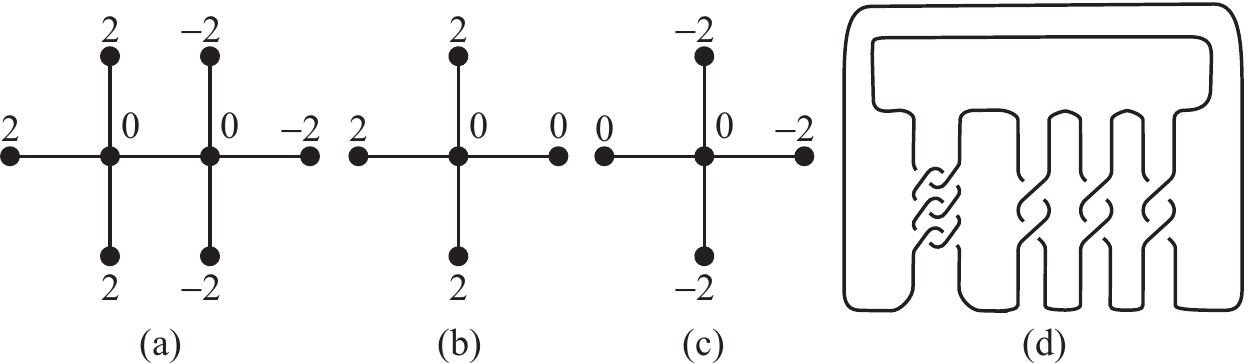} 
\caption{The weighted tree $(\Ts,w)$ in (a) has no sticks 
and two stars, shown in (b) and (c).  Both stars are positive;
but the canonical arborescent link diagram of 
$(\Bd\stripplumb{\Ts,w},S^3)$, shown in (d), has no 
projective orientation making it a positive diagram.
\label{fig:adjacent 0-nodes}}
\end{center}
\end{figure}

\begin{theorem}\label{thm:very stqp weighted trees}
The following are equivalent.  
\begin{inparaenum}[\upshape(A)]
\item
$(\Ts,w)$ is very strongly quasipositive.
\item
The stick of every non-node in $(\Ts,w)$ and the star of 
every node in $(\Ts,w)$ is very strongly positive.
\item
For every $v\in\Verts{\Ts}$, $w(v)$ is even and strictly negative.
\end{inparaenum}
\end{theorem} 

\begin{proof}
Immediate from \ref{lemma:positive sticks and stars}%
\eqref{sublemma:vstqp sticks},
\ref{lemma:positive sticks and stars}%
\eqref{sublemma:positive negative non-pretzel stars},
and Theorem~\ref{thm:qp Murasugi sum theorem}.
\end{proof}

Theorems~\ref{thm:positive weighted trees} and, 
especially, \ref{thm:very stqp weighted trees},
give an adequate account of strongly quasipositive
arborescent links that are constructed from either
positive weighted trees or very strongly quasipositive
weighted trees.  The situation is less satisfactory 
for strongly quasipositive arborescent links that are
of neither of those types: the sufficient conditions
to be described shortly are by no means clearly necessary.
 
Let $S_0$, $S'_0$, and $S_1$ be compact 
$2$--submanifolds-with-boundary of $S^3$ with 
$S_0$ unoriented and $\Bd S_0=\Bd S'_0$ as 
unoriented $1$--manifolds; either or both of
$S'_0$ and $S_1$ may be oriented (if orientable). 
Let $P_0\sub S_0$ and $P_1\sub S_1$ be 
$2p$--gonal plumbing patches.

\begin{definition}\label{def:transplant}
$P_0$ can be \bydef{transplanted} to a $2p$--gonal
plumbing patch $P'_0\sub S'_0$ in case there is an 
ambient isotopy between the $2$--complexes 
$\Bd S_0\cup P_0$ 
and $\Bd S'_0\cup P'_0$ 
respecting the components 
of $\Bd S_0$ and $\Bd S'_0$
and some (equivalently, every) projective orientation 
of $\Bd S=\Bd S'$.  \done
\end{definition}

\begin{lemma}\label{lemma:transplantation}
If $P_0\sub S_0$ can be transplanted to $P'_0\sub S'_0$,
then to any plumbing $S_0\splumb{P_0}{P_1}S_1$ 
corresponds a plumbing $S'_0\splumb{P_0}{P_1}S_1$ 
such that $\Bd(S_0\splumb{P_0}{P_1}S_1)$ and 
$\Bd(S'_0\splumb{P_0}{P_1}S_1)$ are ambient isotopic 
by an isotopy respecting any pre-assigned orientations
of $\Bd(S_0)=\Bd(S'_0)$ and $\Bd(S_1)$ consistent with (say)
the plumbing $\Bd(S_0\splumb{P_0}{P_1}S_1)$.\qed
\end{lemma}

\begin{figure}[ht!]
\begin{center}
\includegraphics[width=\figwidth]{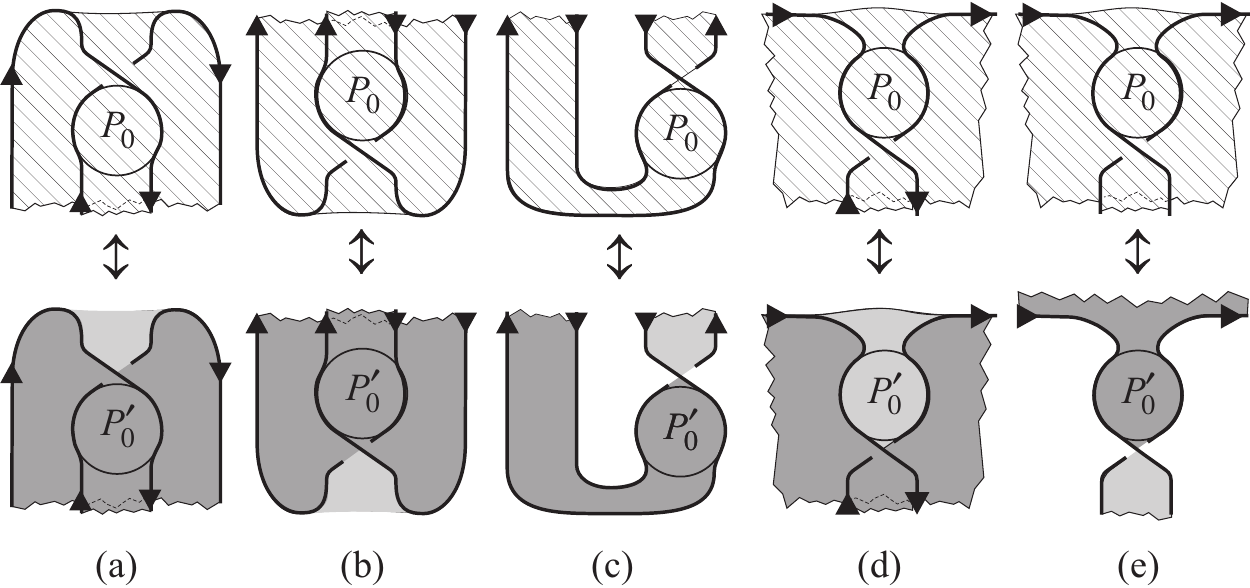} 
\caption{Transplanting a core-transverse plumbing patch 
from an unoriented strip-plumbed surface: the donor, which
is $\stripplumb{\stickgraph{r_1,\dots,r_n}}$ in (a)--(c)
and $\stripplumb{\stargraph{0}{t_1,\dots,t_p}}$ in (d) and (e),
may not be globally orientable (although the pictured part
of it is);
the recipient is a Seifert surface for the donor's boundary 
with the indicated orientation in (a)--(d), and the 
pretzel surface $\pretzelsurface{t_1,\dots,t_p}$ in (e), 
where nothing is oriented.
\label{fig:transplanting plumbing patches}}
\end{center}
\end{figure}

In all instances of transplanting used here, $p=2$; Figure \ref{fig:transplanting plumbing patches} shows them as follows.  
\begin{compactenum}[(1)]
\item
The upper portion of each sub-figure depicts part of one 
strip of an unoriented iterated strip-plumbed 
$2$--manifold-with-boundary $S_0$. 
	\begin{enumerate}[(i)]
	\item
	In (a), $S_0=\stripplumb{\stickgraph{r_1,\dots,r_n}}$
	and the strip is $\strip{\Oscr}{r_1}\sub S_0$.
	\item
	In (b) and (c), $S_0=\stripplumb{\stickgraph{r_1,\dots,r_n}}$
	and the strip is $\strip{\Oscr}{r_n}\sub S_0$.
	\item
	In (d) and (e), $S_0=\stripplumb{\stargraph{c}{t_1,\dots,t_p}}$
	and the strip is any $\strip{\Oscr}{t_i}\sub S_0$.
	\end{enumerate}
\item
In each sub-figure $P_0$ is a core-transverse plumbing patch 
on that strip, and $\Bd S_0$ is equipped with a projective
orientation $\orientation$ making the visible 
crossing positive.
\item
The lower portion of each sub-figure depicts part of 
the Seifert surface $S'_0$ with $\Bd S'_0=\Bd S_0$
that is produced by applying Seifert's algorithm to a 
diagram of $\Bd S_0^\orientation$ extending the partial 
diagram in the sub-figure.
\item 
In each sub-figure, $P'_0$ is $P_0$ transplanted from 
$S_0$ to $S'_0$ (the required isotopy 
can be taken to be constant).  
\end{compactenum}

Figure \ref{fig:plumbing along transplanted plumbing patches}
illustrates \ref{lemma:transplantation} using the surfaces
in Figure \ref{fig:transplanting plumbing patches}(c) and (d).

\begin{figure}[ht!]
\begin{center}
\includegraphics[width=\figwidth]%
{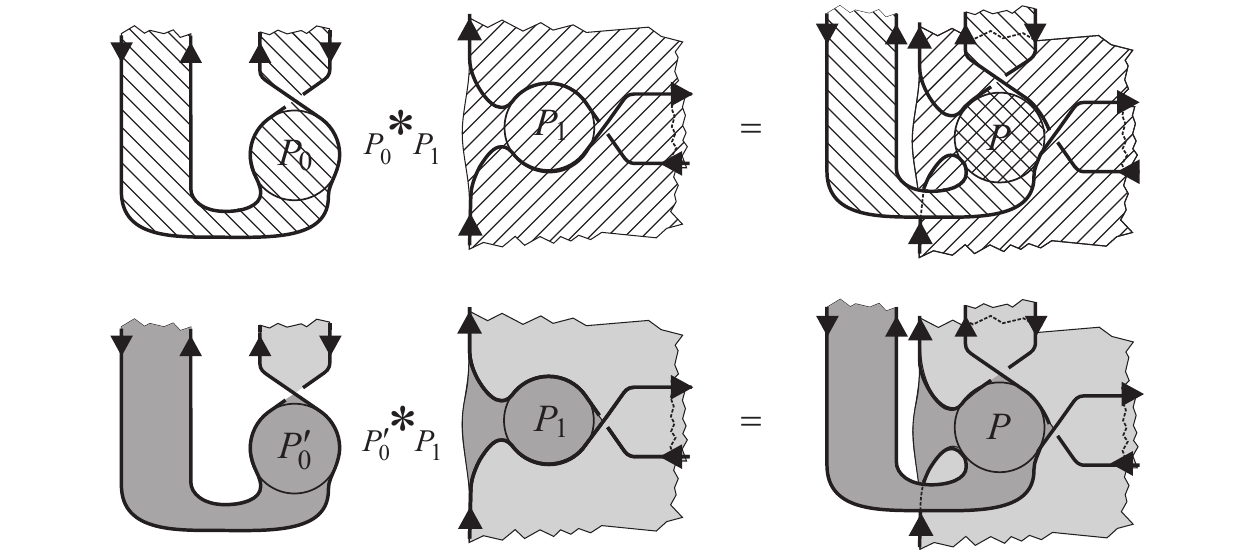} 
\caption{Plumbing along transplanted plumbing patches.
\label{fig:plumbing along transplanted plumbing patches}}
\end{center}
\end{figure}

Note that in the cases illustrated in 
Figure \ref{fig:transplanting plumbing patches}(a)--(d), 
the unique projective orientation of $\Bd P_0$ is consistent
with the given projective orientation of $\Bd S_0$.
Contrariwise, in the remaining cases of first and last 
strips on $\stripplumb{\stickgraph{r_1,\dots,r_n}}$
and any strip $\strip{\Oscr}{t_i}$ 
on $\stripplumb{\stargraph{0}{t_1,\dots,t_p}}$, with
boundaries oriented to make the visible crossing positive,
these projective orientations are inconsistent, and
thus a core-transverse plumbing patch $P_0$ on that 
strip \emph{cannot} be transplanted to a plumbing patch 
on $S'_0$ (by any isotopy whatever); see
Figure~\ref{fig:non-transplantable plumbing patches}.

\begin{figure}[ht!]
\begin{center}
\includegraphics[width=\figwidth]{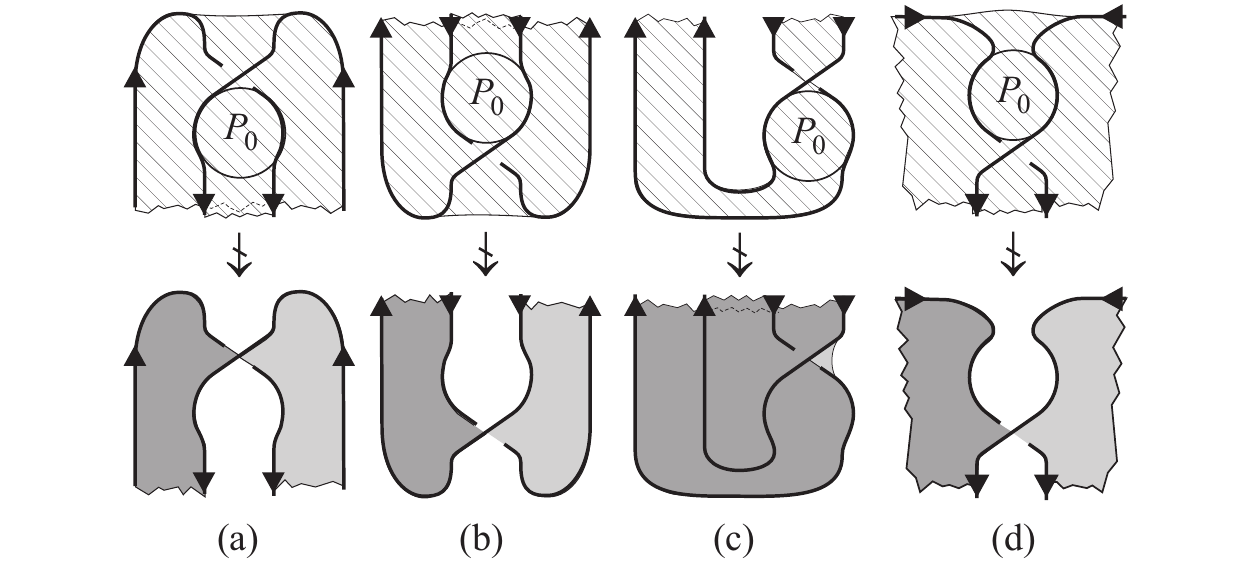} 
\caption{The indicated core-transverse plumbing patches on 
strip-plumbed surfaces cannot be transplanted to 
the indicated Seifert surfaces bounded by the same links. 
\label{fig:non-transplantable plumbing patches}}
\end{center}
\end{figure}

\begin{proposition}%
\label{prop:plumbing to stqp rational/pretzel links}
Let $r_1,\dots,r_n$, $t_1,\dots,t_p$, and $c$ be integers.
\begin{asparaenum}[\upshape(A)]
\item\label{subprop:transplantable ends of positive sticks}
Let $S_0=\stripplumb{\stickgraph{r_1,\dots,r_n}}$.
If $\Bd S_0$ has a projective orientation $\orientation$ for
which $((\Bd S_0)^\orientation,S^3)$ is a positive link, 
$S'_0$ is the quasipositive Seifert surface with 
$\Bd S'_0=(\Bd S_0)^\orientation$ produced by Seifert's 
algorithm (see \textup{\ref{thm:positive implies stqp}}),
and $i=1$ or $i=n$,
then a core-transverse $4$--patch 
$P_0\sub \strip{\Oscr}{r_i}\sub S_0$ can be transplanted to
$S'_0$ iff $r_i<0$.
\item\label{subprop:transplantable strips of positive stars}
Let $S_0\isdefinedas\stripplumb{\stargraph{c}{t_1,\dots,t_p}}$.
If $\Bd S_0$ has a projective orientation $\orientation$ for which
$((\Bd S_0)^\orientation,S^3)$ is a positive link, 
$S'_0$ is the quasipositive Seifert surface with 
$\Bd S'_0=(\Bd S_0)^\orientation$ produced by Seifert's 
algorithm, 
and $1\le i\le p$,
then a core-transverse $4$--patch 
$P_0\sub\strip{\Oscr}{t_i}\sub S_0$ can be transplanted to
$S'_0$ iff $t_i<0$.
\item\label{subprop:transplantable strips of non-qp, sqp stars}
If $c=0$, all $t_i$ have the same parity,
$t_i+t_j<0$ for $1\le i<j\le p$, 
and $1\le i\le p$,
then a core-transverse $4$--patch 
$P_0\sub\strip{\Oscr}{t_i}\sub S_0$ can be transplanted to $S'_0$.
\end{asparaenum}
\end{proposition}
\begin{proof}
\eqref{subprop:transplantable ends of positive sticks} 
follows from \ref{prop:2-bridge machine} applied to
Figure~\ref{fig:transplanting plumbing patches}(a)--(c),
\eqref{subprop:transplantable strips of positive stars}
from \ref{prop:positive pretzel links} applied to
Figure~\ref{fig:transplanting plumbing patches}(d),
and \eqref{subprop:transplantable strips of non-qp, sqp stars}
from \ref{prop:qp pretzel surface}%
\eqref{subprop:pretzel surface negative sum condition}
applied to 
Figure~\ref{fig:transplanting plumbing patches}(e).
\end{proof}

\begin{remark}%
\label{thm:stqp arborescent links by transplanted plumbings}
As suggested by 
Figure~\ref{fig:non-transplanted diagrammatic plumbing example},
\begin{figure}[ht!]
\centering
\includegraphics[width=\figwidth]%
{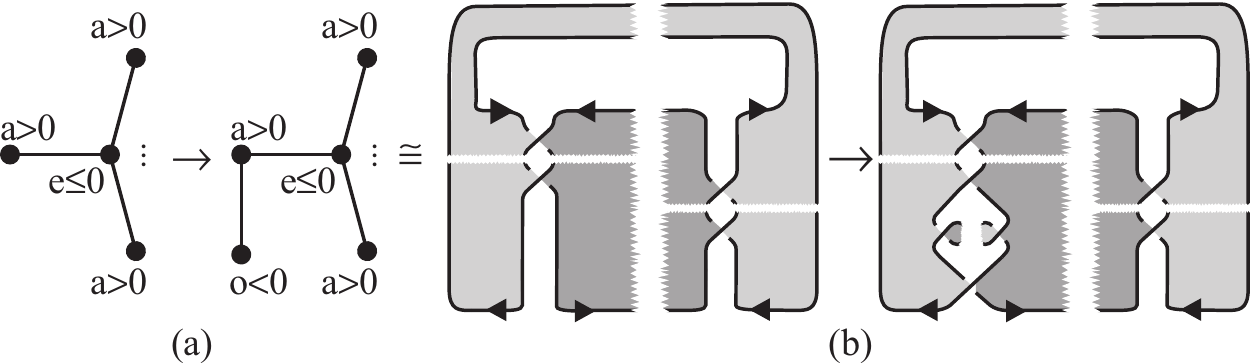} 
\caption{The illustrated strip-plumbing preserves diagrammatic
positivity and therefore strong quasipositivity, but 
cannot be performed by plumbing quasipositive Seifert surfaces.
\label{fig:non-transplanted diagrammatic plumbing example}}
\end{figure}
many arborescent links have positive (and thus  
strongly quasipositive) orientations but are not plumbed 
from strongly quasipositive sticks and stars as in
\ref{thm:positive weighted trees}.  It would
be interesting to have necessary and sufficient
conditions for an arborescent link to have a (strongly)
quaspositive orientation.
\done
\end{remark}

\section{Constructions of $\mathbf{3}$--dimensional 
transverse $\C$--links\label{part:3d constructions}}
This part collects constructions of $3$--dimensional
transverse $\C$--links; many of the 
constructions use quasipositive links constructed 
in Part \ref{part:qp stuff old and new}.

\begin{notation}\label{notations:3-dim stuff}
Let $\Delta\sub U\sub\C^{2}$, $\Sigma=\Bd{\Delta}$, 
and $f\in\holo{U}$ be as in Part \ref{part:introduction}. 
\begin{inparaenum}[(1)]
\item\label{notation:product with smoothed corners}
For $r>\max\{|f(\zz)|\setsuchthat \zz\in\Delta\}$,
the product $\Delta\times\DSK{2}{0}{r}\sub 
U\times\C\sub\C^{3}$ is a closed Stein $6$--disk.  
Its boundary $\Bd(\Delta\times\DSK{2}{0}{r})=%
\Sigma\times \DSK{2}{0}{r}\cup\Delta\times\SPH{1}{0}{r}$,
a piecewise real-analytic $5$--sphere, is  
pseudoconvex but not strictly so; however, 
$\Delta\times\DSK{2}{0}{r}$ can be arbitrarily well 
approximated by closed Stein $6$--disks in 
$U\times\C$ with strictly pseudoconvex real-analytic
boundaries.  Write $\pad{\Delta}$ for such an
approximation that is sufficiently close for whatever 
purpose is required, and $\pad{\Sigma}$ for 
$\Bd\pad{\Delta}$. 
\item\label{notation:[q]}
For an integer $q>0$, define $\cycsusp{f}{q}$ by
$(\cycsusp{f}{q})(z_0,z_1,z_2)=f(z_0,z_1)+z_2^q$; 
define $\cycsusp{f}{0}$ by
$(\cycsusp{f}{0})(z_0,z_1,z_2)=f(z_0,z_1)$.
For a $1$--dimensional transverse $\C$--link
$\Clink{f}{\Sigma}$, write 
$\cycsusp{\Clink{f}{\Sigma}}{q}\isdefinedas
(V(\cycsusp{f}{q}),\pad{\Sigma})$; for
$q>0$, this is an instance of what Kauffman and Neumann
\cite{KauffmanNeumann1977} call the \bydef{$q$-fold cyclic 
suspension} $\cycsusp{\Lscr}{q}$ of a 
smooth, oriented link $\Lscr$.\done 
\end{inparaenum}
\end{notation}

\subsection{$\mathbf{3}$--dimensional links of isolated 
singular points}\label{subsect:links of singularities}
This heading is included for completeness only, since the
theory of these $3$--dimensional transverse
$\C$--links has been thoroughly developed (\cite{Milnor1968},
\cite{Neumann1981}, 
\cite{EisenbudNeumann1985}, etc).

\subsection{Adding a dummy variable}\label{subsect:dummy variable}
\begin{proposition}\label{prop:dummy of transverse}
If $\Clink{f}{\Sigma}$ is a $1$--dimensional transverse
$\C$--link with non-singular $\C$--span $\Cspan{f}{\Sigma}$, then:
\begin{inparaenum}[\upshape(1)]
\item\label{subprop:dummy is a C-link} 
$\cycsusp{\Clink{f}{\Sigma}}{0}=
\Clink{\cycsusp{f}{0}}{\pad{\Sigma}}$
is a transverse $3$--dimensional $\C$--link; 
\item\label{subprop:Cspan of dummy C-link} 
$\Cspan{\cycsusp{f}{0}}{\pad{\Delta}}$ is non-singular,
and diffeomorphic to a disjoint union of boundary-connected 
sums of copies of $S^1\times D^3$;
\item\label{subprop:linkman of dummy C-link}
$\lman{\cycsusp{f}{0}}{\pad{\Sigma}}$ is diffeomorphic 
to a disjoint union of connected sums of copies of 
$S^1\times S^2$.
\end{inparaenum}
\end{proposition}

\begin{proof}
Both \eqref{subprop:dummy is a C-link} and 
\eqref{subprop:linkman of dummy C-link} follow from 
\eqref{subprop:Cspan of dummy C-link}.
To see \eqref{subprop:Cspan of dummy C-link}, 
note that the $2$--manifold-with-boundary 
$\Cspan{f}{\Delta}$ has a handle decomposition 
into $2$--dimensional $0$--handles and $1$--handles
attached orientably to the $0$--handles; therefore
the product $\Cspan{f}{\Delta}\times D^2$ (with 
corners smoothed) has a handle decomposition into 
$4$--dimensional $0$--handles and $1$--handles 
attached orientably to the $0$--handles,
and so must be as described.
\end{proof}

\ref{prop:dummy of transverse} yields particularly
interesting examples in case $\Cspan{f}{\Delta}$ is a 
$2$--disk, so that $\Clink{f}{\Sigma}$ is a slice 
knot (in fact a ribbon knot; cf \cite{Rudolph1983}).  

\begin{proposition}
\label{prop:slice S3s}
If $\Clink{f}{S^3}$ is a quasipositive slice knot, then
the $3$--dimensional transverse $\C$--link 
$\Clink{\cycsusp{f}{0}}{\pad{S^3}}$ is a slice
knot in the $5$--sphere $\pad{S^3}$.\qed
\end{proposition}

\begin{questions}
There are infinitely many pairwise non-isotopic quasipositive
slice knots in $S^3$ (indeed, there are infinitely many of
braid index $3$).  
\begin{inparaenum}[(1)]
\item
Are there infinitely many pairwise
non-isotopic slice $3$-dimensional transverse $\C$--links 
in $S^5$?  
\item
Specifically, if $\Clink{f_0}{S^3}$ and $\Clink{f_1}{S^3}$ are 
non-isotopic quasipositive slice knots in $S^3$,
are the $3$--dimensional transverse $\C$--links
$\Clink{\cycsusp{f}{0}}{\pad{S^3}}$ non-isotopic? 
\end{inparaenum}
\end{questions}

\subsection{General cyclic branched covers of $S^3$
over quasipositive links%
\label{subsect:cyclic branched covers (general)}}
The $q$--fold cyclic suspension $\cycsusp{\Lscr}{q}$ of a link
$\Lscr=(L,S^m)$, introduced by Neumann \cite{Neumann1974}
and Kauffman and Neumann \cite{KauffmanNeumann1977}, was
defined in \ref{notations:3-dim stuff}\eqref{notation:[q]}
in the special case that $\Lscr=\Clink{f}{S^3}$ is a 
$1$--dimensional transverse $\C$--link.  Cyclic 
suspensions of arbitrary links are themselves special 
cases of what Kauffman \cite{Kauffman1974}
and Kauffman and Neumann \cite{KauffmanNeumann1977} call
the \bydef{knot product} $\Kscr \otimes \Lscr$ of links
$\Kscr$ and $\Lscr$.  In a general knot product
$\Kscr \otimes \Lscr$,  
$\Kscr=(K,S^k)$ is any smooth, oriented $(k-2)$--dimensional
link, $\Lscr=(L,S^\ell)$ is a fibered smooth, oriented
$(\ell-2)$--dimensional link, and $\Kscr \otimes \Lscr$ 
is a smooth, oriented $(k+\ell-1)$--dimensional link 
$(K \otimes L,S^{k+\ell+1})$.

\begin{theorem}[Kauffman and Neumann \cite{KauffmanNeumann1977}]
\label{thm:KN generalities}
\begin{inparaenum}[\upshape(1)]
\item\label{subthm:link-manifold of cyclic suspension}
The link-manifold of the $q$--fold cyclic suspension
of $\cycsusp{\Lscr}{q}$ of a link $\Lscr=(L,S^m)$ is
the $q$--fold cyclic branched cover of $S^m$ branched along $L$.
\item\label{subthm:knot product of fibered links}
If both $\Kscr$ and $\Lscr$ are fibered, then so is 
$\Kscr \otimes \Lscr$; in particular, the $q$--fold 
cyclic suspension of a fibered link is fibered.\qed
\end{inparaenum}
\end{theorem}

Whatever is not obvious in the next proposition
follows directly from Theorem \ref{thm:KN generalities}. 

\begin{proposition}\label{prop:KN for qp links}
Let $q\ge 1$.  
\begin{inparaenum}[\upshape(A)]
\item\label{subprop:cycsusp of qp link} 
If $\Clink{f}{S^3}$ is a $1$--dimensional 
transverse $\C$--link, then:
	\begin{inparaenum}[\upshape(1)]
	\item\label{subsubprop:cycsusp of qp link is transverse link}
	its $q$--fold cyclic suspension 
	$\cycsusp{\Clink{f}{S^3}}{q}=\Clink{\cycsusp{f}{q}}{\pad{S^3}}$
	is a $3$--dimensional transverse $\C$--link;
	\item\label{subsubprop:link-manifold of cycsusp of qp link}
	the link-manifold $\lman{\cycsusp{f}{q}}{\pad{S^3}}$ is 
	the $q$--fold cyclic branched cover of $S^3$ branched along
	$\lman{f}{S^3}$;
	\item\label{subsubprop:Cspan of cycsusp of qp link}
	the $\C$--span $\Cspan{\cycsusp{f}{q}}{\pad{D^4}}$ is the
	$q$--fold cyclic branched cover of $D^4$ branched along 
	$\Cspan{f}{D^4}$.
	\end{inparaenum}
\item\label{subprop:cycsusp of variously special qp links}
If in addition $\Clink{f}{S^3}$ is 
	\begin{inparaenum}[\upshape(1)]
	\item\label{subsubprop:cycsusp of fibered qp link}
	 fibered,
	\item\label{subsubprop:cycsusp of isolated singular point}
	the link of an isolated singular point, or
	\item\label{subsubprop:cycsusp of link at infinity}
	the link at infinity of a polynomial,
	\end{inparaenum}
then 
$\Clink{\cycsusp{f}{q}}{\pad{S^3}}$ is of the same type.\qed
\end{inparaenum}
\end{proposition}

\begin{remark}\label{rmk:Harvey et al and Baldwin}
Let $M$ be the $q$--fold cyclic branched cover of $S^3$ 
branched along a quasipositive link.  
Harvey, Kawamuro, and Plamenevskaya \cite{Harveyetal2009} 
have used contact topology to find a Stein-fillable contact 
structure $\xi$ on $M$.  Proposition \ref{prop:KN for qp links}%
\eqref{subprop:cycsusp of qp link}%
\eqref{subsubprop:link-manifold of cycsusp of qp link},
together with the fact about Stein fillings noted just 
after Theorem \ref{thm:Bennequin's theorem}, gives the 
(apparently) stronger conclusion that $\xi$ can be
required to have a Stein filling by a Stein 
domain on a complex algebraic surface in $\C^3$.
\done
\end{remark}

\subsection{Double branched covers of $S^3$ over 
quasipositive links}\label{subsect:double branched covers}
It is traditional to call $2$--fold branched covers
\bydef{double} branched covers.  Double branched covers 
have two useful properties that distinguish them among 
all cyclic branched covers of classical links.

\begin{theorem}\label{thm:double cover properties}
If $\Lscr=(L,S^3)$ is an oriented classical link, 
then the double cover of $S^3$ branched over $L$ is
invariant under both changes of orientation of $\Lscr$
and mutation of $\Lscr$.
\end{theorem}
\begin{figure}[ht!]
\begin{center}
\includegraphics[width=\figwidth]{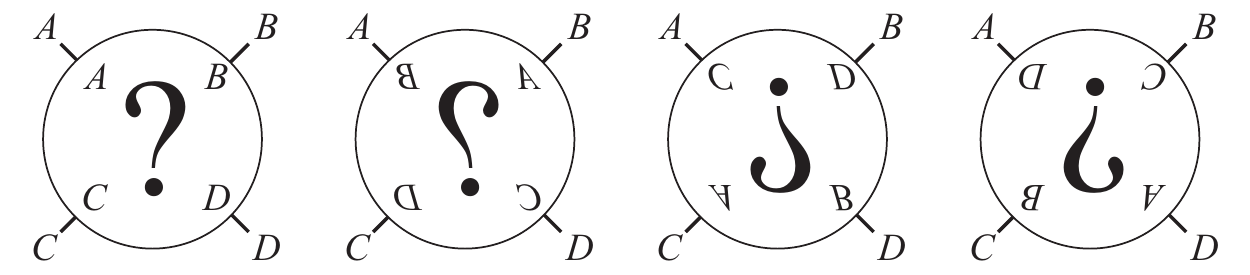} 
\caption{Part of a diagram for an unoriented classical 
link-manifold $L$ is shown schematically at the left; the 
question mark stands for an arbitrary tangle with four 
endpoints $A,\dots,D$ (its two strings may be knotted, 
and it may have simple closed curve components).  By leaving 
alone what is not shown while replacing the shown piece 
with one of its three transforms (at the right), 
the original diagram is transformed into a diagram of 
an \bydef{elementary mutation} of $L$.
\label{fig:mutations}}
\end{center}
\end{figure}
Here a \bydef{mutation} is the composition of finitely 
many \bydef{elementary mutations} as depicted and described
(for unoriented links, their appropriate setting in
this context) in Figure~\ref{fig:mutations}.
\begin{proof}
Invariance under changes of orientation is trivial
(and vacuously so in case $\Lscr$ is a knot);  
invariance under mutation was proved by Montesinos
\cite{Montesinos1973} and Viro \cite{Viro1976}.
\end{proof}

\begin{remarks}\label{rmk:mutation history}
The operation of mutation was introduced explicitly 
for link diagrams by Conway \cite{Conway1970} and 
explicitly for links themselves by Montesinos and Viro, 
but none of \cite{Conway1970}, \cite{Viro1976}, or
\cite{Montesinos1973} contains the term ``mutation''; 
I do not know when (and by whom) that word was first 
used, and would welcome information on the topic.
\done
\end{remarks}

\subsubsection{Doubles of knot exteriors}%
\label{subsubsect:doubles of knot exteriors}
For any manifold-with-boundary $M$, the \bydef{double} of
$M$ is the (suitably smoothed) identification space
$M\times\{0,1\}/{\sim}$, where the non-trivial equivalence
classes of the equivalence relation ${\sim}$ are precisely the 
pairs $\{(x,0),(x,1)\}$ with $x\in\Bd M$.

\begin{lemma} If $\Kscr$ is a classical knot, then 
the double cover of $S^3$ branched over 
$|\Kscr\{2,0\}|=|\AKn{\Kscr}{0}|$ is diffeomorphic to the 
double of the exterior $E(\Kscr)$ of $\Kscr$.\qed
\end{lemma}

\begin{proposition}\label{prop:double double}
If either 
\begin{inparaenum}[\upshape(a)]
\item\label{subprop:untwisted 2-cable of qp is qp}
$\Kscr$ is quasipositive or
\item\label{subprop:AKn(K,2) is qp if TB(K) is negative}
the maximal Thurston--Bennequin invariant $\TB(\Kscr)$ of
$\Kscr$ is non-negative,
\end{inparaenum}
then the double of $E(\Kscr)$ occurs as the link-manifold 
of a $3$--dimensional transverse $\C$--link $\Clink{f}{\Sigma^5}$.
\end{proposition}
\begin{proof} 
In case \eqref{subprop:untwisted 2-cable of qp is qp}
the conclusion follows from \ref{prop:untwisted cables on qp} 
(with $n=2$) and \ref{prop:KN for qp links} (with $q=2$);
in case \eqref{subprop:AKn(K,2) is qp if TB(K) is negative}
the conclusion follows from \ref{thm:qp annuli} and 
\ref{prop:KN for qp links} (with $q=2$).
\end{proof}

\begin{remarks}\label{rmks:Cspans of doubled knot exteriors}
\begin{inparaenum}[(1)]
\item\label{rmk:Cspan of K(2,0)}
If the quasipositive knot $\Kscr$ is a transverse
$\C$--link with non-singular $\C$--span $S$, 
then by the proof of \ref{prop:untwisted cables on qp}
the $\C$--span of $\Kscr\{2,0\}$ can be taken to be
two parallel copies of $S$.  In this case, 
\ref{prop:KN for qp links}%
\eqref{subsubprop:Cspan of cycsusp of qp link}
implies that the $\C$--span of the $3$--dimensional 
transverse $\C$--link $\Clink{f}{\Sigma^5}$
in \ref{prop:double double} has Euler characteristic 
$2-\chi(S)$.
\item\label{rmk:Cspan of Bd AK0}
If the annulus $\AKn{\Kscr}{0}$ is a quasipositive Seifert 
surface, and the $\C$--span of the strongly quasipositive
link $(\Bd\AKn{\Kscr}{0},S^3)$ is non-singular (which may
be assumed), then that $\C$--span is also an annulus; in
this case, the $\C$--span of the $3$--dimensional 
transverse $\C$--link $\Clink{f}{\Sigma^5}$
in \ref{prop:double double} has Euler characteristic 
$2$.
\item\label{rmk:different Cspans}
Consequently, if $\Kscr\ne\Oscr$ is strongly quasipositive, 
then the double of $E(\Kscr)$ occurs as the link-manifold of
two transverse $\C$--links in $S^5$ with non-singular
$\C$--spans that are not diffeomorphic to each other.
In general the two links should not be expected
to be ambient isotopic to each other.\done
\end{inparaenum}
\end{remarks}

\subsubsection{Seifert manifolds with base $S^2$}%
\label{subsubsect:Seifert spaces}
It is standard (see, eg, Bonahon and Siebenmann 
\cite{BonahonSiebenmann1979/2010}) that 
a $3$--manifold $M$ is Seifert-fibered over $S^2$ with
at least $3$ exceptional fibers iff $M$ is the double 
branched cover of $S^3$ branched over a pretzel link-manifold 
$\pretzellink{t_1,\dots,t_p}$ (the restriction on the number 
of exceptional fibers is an artifact of the definitional 
restriction on pretzel links that $p$ be at least $3$), 
and then (in the language of \cite[p.~323]%
{BonahonSiebenmann1979/2010}) $(0;1/t_1,\dots,1/t_p)$ is 
the \bydef{raw data vector} of the \bydef{Seifert manifold}
$M$.

\begin{lemma}\label{lemma:mutant pretzels}
For any permutation $\pi$ of $\{1,\dots,p\}$, the
pretzel link $\pretzel{t_{\pi(1)},\dots,t_{\pi(p)}}$
is a mutation of $\pretzel{t_1,\dots,t_p}$ \qed
\end{lemma}

By \ref{thm:double cover properties} and
\ref{lemma:mutant pretzels}, in the description
of $M$ as the double branched cover of $S^3$ over
$\pretzellink{t_1,\dots,t_p}$ no generality is 
lost by requiring
\begin{equation}\label{eqn:pretzel in good order}
p=q+r+s, \qua t_1,\dots,t_q>1,\qua t_{q+1},\dots,t_{q+r}<-1, 
\qua \abs{t_{q+r+1}},\dots,\abs{t_p}=1.
\end{equation}
On assumption \eqref{eqn:pretzel in good order}, 
the \bydef{Seifert data vector} of $M$ (again following 
\cite{BonahonSiebenmann1979/2010}) is
\begin{equation*} 
(0;-r;{1}/{t_1},\dots,{1}/{t_q},
1+{1}/{t_{q+1}},\dots,1+{1}/{t_{q+r}})
\end{equation*}
and $M$ has the representation
\begin{equation}\label{eqn:Seifert notation} 
M(\mathrm{O},\mathrm{o};0;-r;(t_1,1),\dots,(t_q,1),
(-t_{q+1},-1-t_{q+1}),\dots,(-t_{q+r},-1-t_{q+r}))
\end{equation}
in (essentially) the original notation of Seifert 
\cite{Seifert1933}.  

Note that $M$ does not depend on $s$, so that the double 
cover of $S^3$ branched over 
\begin{equation*}
\pretzellink{t_1,\dots,t_{q+r},
\overbrace{1,\dots,1}^{\text{$s_{+}$ times}},
\overbrace{-1,\dots,-1}^{\text{$s_{-}$ times}}}
\end{equation*}
is independent of $s_{+}$ and $s_{-}$, 
although (with trivial exceptions) the links 
corresponding to given values of $s_{+}$ and $s_{-}$
are ambient isotopic (and mutations of each other)
iff they have the same value $s_{+}-s_{-}$.
  
\begin{proposition}\label{prop:C-transverse Seifert manifolds}
If $(t_1,\dots,t_p)$ is a $p$--tuple of integers
satisfying \textup{\eqref{eqn:pretzel in good order}}, 
then the Seifert manifold \textup{\eqref{eqn:Seifert notation}}
is the link-manifold of a $3$--dimensional transverse 
$\C$--link in each of the following cases.
\begin{inparaenum}[\upshape(A)]
\item\label{subprop:Seifert manifolds, all odd neg}
All $t_i$ are odd and negative.
\item\label{subprop:Seifert manifolds, no odd neg}
No $t_i$ is odd and negative, and an even number of
$t_i$ are positive.
\item\label{subprop:Seifert manifolds, all even, neg sums}
All $t_i$ are even, and $t_i+t_j<0$ for $1\le i<j\le p$.
\item\label{subprop:Seifert manifolds, special slice pretzels}
$(t_1,\dots,t_p)= (2n+1,-(2n+1),-2m)$ for $m, n>0$.
\end{inparaenum}
\end{proposition}

\begin{proof}
\eqref{subprop:Seifert manifolds, all odd neg}
and \eqref{subprop:Seifert manifolds, no odd neg}
follow from \ref{prop:positive pretzel links},
\eqref{subprop:Seifert manifolds, all even, neg sums}
from \ref{prop:qp pretzel surface}%
\eqref{subprop:pretzel surface parity condition},
and \eqref{subprop:Seifert manifolds, special slice pretzels}
from \ref{qn:qp pretzel questions}%
\eqref{subqn:non-strongly qp pretzel links}, all upon
passing to double covers of $S^3$ branched over the 
relevant quasipositive links.
\end{proof}

Gompf \cite{Gompf1998} shows that if $M$ is a Seifert 
manifold then $M$, at least one of $M$, $\Mir M$ has a Stein filling.   
\ref{prop:C-transverse Seifert manifolds} allows one to 
find such fillings that lie on algebraic surfaces in $\C^3$, 
and to calculate knot-theoretical properties of transverse 
$\C$--links with $M$ and/or $\Mir M$ as link-manifold.

\begin{example}[Boileau and Rudolph \cite{BoileauRudolph1995}]
For positive $\ell,m,n>0$, let $\Sigma(\ell,m,n)$
denote the $3$--dimensional \bydef{Brieskorn manifold} 
$\lmansing{(0,0,0)}{z_0^\ell+z_1^m +z_2^n}$. 
A calculation following Neumann \cite{Neumann1970}
shows that $\Sigma(\ell, m, n)$ is the Seifert manifold 
fibered over $S^2$ with three exceptional fibers with
$M(\mathrm{O},\mathrm{o};0;-r;(\ell,1),(m,1),(n,1))$
as its Seifert notation.  Suppose that
$\ell m+\ell n-mn=\epsilon\in\{-1,1\}$ and 
$\ell\equiv 1+\epsilon\pmod 2$;
for instance, $(\ell,m,n)$ could be 
$(2t-1,2t+1,2t^2-1)$ or $(2t,2t+1,2t(2t+1)+1)$.
In this situation, \ref{prop:C-transverse Seifert manifolds}  
implies that both $\Sigma(\ell,m,n)$ and 
$\Mir\Sigma(\ell,m,n)$
are link-manifolds of $3$--dimensional transverse $\C$--links.
(This example is due to Michel Boileau.)\done
\end{example}

\subsubsection{Lens spaces}\label{subsubsect:lens spaces}
It is standard (Schubert \cite{Schubert1956}; see
Burde and Zieschang \cite{BurdeZieschang1985}) that a 
$3$--manifold is a lens space (including $S^1\times S^2$) 
iff $M$ is the double branched cover of $S^3$ branched over 
a rational link-manifold $\rationallink{r_1,r_2,\dots,r_n}$,
and then $M=L(P,Q)$ where 
\begin{equation*}
\frac{P}{Q}
\isdefinedas
r_1+
\cfrac{1}{-r_2+
\cfrac{1}{\dotsb+
\cfrac{1}{(-1)^{n-1}r_n
}}}
\end{equation*}
and $P>0$ is relatively prime to $Q$.  By 
\ref{prop:2-bridge machine} and 
\ref{prop:KN for qp links},
we have the following.

\begin{proposition}\label{prop:C-transverse lens spaces}
With $P$, $Q$ as above, 
if $\sigma_2^{r_1}\sigma_3^{r_2}\sigma_3^{r_3}\cdots%
\sigma_\ell^{r_n}\in B_4$ (with $\ell$ equal to $3$ or $2$ 
according as $n$ is even or odd) is generated by the 
labeled digraph in Figure \textup{\ref{fig:2-bridge machine}},
then $L(P,Q)$ is the link-manifold of a $3$--dimensional 
transverse $\C$--link.\qed
\end{proposition}

\begin{example}[\cite{BoileauRudolph1995}]
If $p,q>1$ are odd integers, then the lens spaces 
$L(pq+1,p)$ and $\Mir L(pq+1,p)$ both appear as 
link-manifolds of $3$--dimensional transverse
$\C$--links.  (This example is due to Michel Boileau.)\done
\end{example}

\subsubsection{Tree-manifolds}
\label{subsubsect:tree-manifolds}
Let $(\Ts,w)$ be a weighted planar tree. The double cover 
$M^3(\Ts,w)$ of $S^3$ branched over the arborescent link
$(\Bd\sp{\Ts,w})$ is called a \bydef{tree-manifold}; it 
is independent of the planar embedding of $\Ts$.   

\begin{remark}\label{rmk:graph manifold history}
Tree-manifolds are a special case of \bydef{graph-manifolds}.
Graph-manifolds can be defined in various 
(not obviously equivalent) ways; they were named
and first investigated in full generality by 
Waldhausen \cite{Waldhausen1967a,Waldhausen1967b}.
Waldhausen's work built on studies of tree-manifolds
by Hirzebruch \cite{Hirzebruch1953} and 
von Randow \cite{Randow1962}.  For them, 
the tree-manifold $M^3(\Ts,w)$ arises as the
boundary of a $4$--manifold $W^4(\Ts,w)$ constructed
by \bydef{$4$--dimensional plumbing of disk bundles}. 
Hirzebruch, Neumann, and Koh 
\cite{HirzebruchNeumannKoh1971} give a further
exposition of tree-manifolds from this viewpoint.
Neumann \cite{Neumann1981} gives a calculus for 
plumbing trees that is simultaneously
applicable to strip-plumbings $\stripplumb{\Ts,w}$
of unoriented $2$--submanifolds-with-boundary of $S^3$,
disk-bundle plumbings $W^4(\Ts,w)$, and 
tree-manifolds $M^3(\Ts,w)\isdefinedas\Bd W^4(\Ts,w)$.
\done
\end{remark} 

\begin{proposition}\label{prop:C-transverse tree-manifolds}
If $(\Ts,w)$ is strongly quasipositive, then
$M^3(\Ts,w)$ is the link-manifold of a $3$--dimensional
transverse $\C$--link; if $(\Ts,w)$ is very strongly
quasipositive, then $W^4(\Ts,w)$ is the $\C$--span
of a $3$--dimensional transverse $\C$--link.\qed
\end{proposition}

This is a considerable strengthening of the following
result, stated without proof (and using slightly different
language) by Boileau and Rudolph \cite{BoileauRudolph1995} 
in 1995.

\begin{corollary}
\label{cor:original B-R theorem}
Let the weighted tree $(\Ts,w)$ satisfy the following 
conditions.
\begin{inparaenum}[\upshape(1)]
\item\label{oldthm:hyp1}
If $v\in\Verts{\Ts}$ is neither a node nor adjacent to a 
node, then $w(v)$ is even and less than $0$.
\item\label{oldthm:hyp2}
If $v\in\Verts{\Ts}$ is a node, then $w(v)$ is 
even and not greater than $0$.
\item\label{oldthm:hyp3} 
Let $v\in\Verts{\Ts}$ be a node with adjacent vertices $v'_i$, 
$1\le i \le r$. 
	\begin{inparaenum}[\upshape(a)]
	\item 
	If $w(v)<0$, then $w(v_i)$ is even and less than
	$0$, $1\le i \le r$. 
	\item 
	If $w(v)=0$, then $v$ is distant from every 
	other node, and $w(v_i)+w(v_j)$ is even and less than
	$0$, $1\le i<j \le r$.
	\end{inparaenum}
\end{inparaenum}
Then there is a projective orientation of the arborescent 
link $(\Bd\stripplumb{\Ts,w},S^3)$ that makes it a 
quasipositive link.\qed
\end{corollary}

\subsection{$\mathbf{3}$--dimensional links at infinity of complex
surfaces\label{subsect:links at infinity}}

Using algebraic topology, Sullivan \cite{Sullivan1975} 
proved that the link-manifold $M$ of an isolated singular 
point of a complex algebraic surface in $\C^3$ (actually, 
and more generally, the link---in the older sense of 
combinatorial topology, not that of knot-theory---of 
an isolated singular point of a complex algebraic 
surface in any $\C^n$) cannot be diffeomorphic to 
the $3$--torus $(S^1)^3$.  On the other hand, if 
$f(z_0,z_1,z_2)=z_0 z_1 z_2 - 1$, then 
$\Clinkinfty{f}$ has link-manifold diffeomorphic to
$(S^1)^3$ and $\C$--span diffeomorphic
to $(S^1)^2\times D^2$; the proof consists
in observing that, for sufficiently
small $r>0$, $(S^1)^2$ acts freely on
\begin{equation*}
\{(z_0,z_1,z_2)\in\C^3\setsuchthat f(z_0,z_1,z_2)=0,\qua
z_0^2+z_1^2+z_2^2 \le 1/r^2\}
\end{equation*}
by $(\mathrm{e}^{i\theta},\mathrm{e}^{i\phi})
\cdot(z_0,z_1,z_2)=
(\mathrm{e}^{i\theta}z_0,
\mathrm{e}^{i\phi}z_1,\mathrm{e}^{-i(\theta+\phi)z_2})$, 
and the slice 
\begin{equation*}
\{(x_0,x_1,x_2)\in\R_{+}^3\setsuchthat 
f(x_0,x_1,x_2)=0,\qua
x_0^2+x_1^2+x_2^2 \le 1/r^2\}
\end{equation*}
of this action is diffeomorphic to $D^2$. 

In fact, all the products $S^1\times F_g$ (where 
$F_g$ is the closed orientable $2$--manifold of genus $g$) 
arise as link-manifolds of links at infinity:
for sufficiently small $\epsilon>0$, 
$\{(z_0,z_1,z_2)\in\C^3: z_0 z_1(z_2^g-1)=1, 
\abs{z_0}^2+\abs{z_1}^2+\abs{z_2xi}^2=1/\epsilon \}$
is diffeomorphic to $S^1\times F_g$ (Boileau and Rudolph \cite{BoileauRudolph1995}).

\section*{Acknowledgments}

Many thanks to R.\ Inanc Baykur, John Etnyre, and Ursula 
Hamenst\"adt, the organizers of the conference on ``Interactions 
between low dimensional topology and mapping class groups'' 
(MPI--Bonn, July 2013), for their invitation to speak there
and to contribute to that volume.  I also thank an
anonymous referee for a close reading and numerous 
helpful suggestions, including references that I
had overlooked.   Preliminary versions 
of some results in part \ref{part:3d constructions},
particularly portions of   
\ref{subsubsect:Seifert spaces}--\ref{subsubsect:tree-manifolds}
and \ref{subsect:links at infinity}, were 
discovered in collaboration with Michel Boileau; I am very 
grateful for that collaboration, and assume all responsibility 
for our failure (so far) to publish a final version of 
\cite{BoileauRudolph1995}.

%
%
%
\bibliographystyle{gtart}
\bibliography{C3mflds}

%

\end{document}